\documentclass[10pt]{amsart}
\usepackage{a4wide}
\usepackage[latin9]{inputenc}
\usepackage{amsmath}
\usepackage{amsfonts}
\usepackage{amssymb}
\usepackage{amsthm}
\usepackage{subfigure}
\newtheorem{theo}{Theorem}[section]
\newtheorem{prop}[theo]{Proposition}
\newtheorem{coro}[theo]{Corollary}
\newtheorem{lemm}[theo]{Lemma}

\theoremstyle{definition}

\theoremstyle{remark}
\newtheorem{rema}[theo]{Remark}
\newcommand{\Op}{\operatorname{Op}}
\newcommand{\U}{\mathcal{U}}
\newcommand{\nwc}{\newcommand}
\nwc{\eps}{\epsilon}
\nwc{\ep}{\epsilon}
\nwc{\vareps}{\varepsilon}
\nwc{\Oph}{\operatorname{Op}_\hbar}
\nwc{\la}{\langle}
\nwc{\ra}{\rangle}

\nwc{\mf}{\mathbf} 
\nwc{\blds}{\boldsymbol} 
\nwc{\ml}{\mathcal} 

\nwc{\defeq}{\stackrel{\rm{def}}{=}}

\nwc{\cE}{\ml{E}}
\nwc{\cN}{\ml{N}}
\nwc{\cO}{\ml{O}}
\nwc{\cP}{\ml{P}}
\nwc{\cU}{\ml{U}}
\nwc{\cV}{\ml{V}}
\nwc{\cW}{\ml{W}}
\nwc{\tU}{\widetilde{U}}
\nwc{\IN}{\mathbb{N}}
\nwc{\IR}{\mathbb{R}}
\nwc{\IZ}{\mathbb{Z}}
\nwc{\IC}{\mathbb{C}}
\nwc{\tP}{\widetilde{P}}
\nwc{\tPi}{\widetilde{\Pi}}
\nwc{\tV}{\widetilde{V}}
\nwc{\supp}{\operatorname{supp}}
\nwc{\rest}{\restriction}

\renewcommand{\Im}{\operatorname{Im}}
\renewcommand{\Re}{\operatorname{Re}}

\begin{document}

\title{Eigenmodes of the damped wave equation and small hyperbolic subsets}

\author[Gabriel Rivi\`ere]{Gabriel Rivi\`ere\vspace{0.5cm}\\
With an appendix by St\'ephane Nonnenmacher and Gabriel Rivi\`ere}

\address{Institut de Physique Th\'eorique (CEA Saclay), Orme des Cerisiers, CEA Saclay, 91191 Gif-sur-Yvette Cedex, France}
\email{snonnenmacher@cea.fr}

\address{Laboratoire Paul Painlev\'e (U.M.R. CNRS 8524), U.F.R. de Math\'ematiques, Universit\'e Lille 1, 59655 Villeneuve d'Ascq Cedex, France}
\email{gabriel.riviere@math.univ-lille1.fr}

\thanks{This work has been partially supported by the grant ANR-09-JCJC-0099-01 of the Agence Nationale de la Recherche.}

\begin{abstract}
We study stationary solutions of the damped wave equation on a compact and smooth Riemannian manifold without boundary. In the high frequency limit, we prove that a sequence of $\beta$-damped stationary solutions cannot be completely concentrated in small neighborhoods of a small fixed hyperbolic subset made of $\beta$-damped trajectories of the geodesic flow. 

The article also includes an appendix (by S. Nonnenmacher and the author) where we establish the existence of an inverse logarithmic strip without eigenvalues below the real axis, under a pressure condition on the set of undamped trajectories. 

\end{abstract}

\maketitle

\section{Introduction}

Let $M$ be a smooth, connected, compact Riemannian manifold of dimension $d\geq 2$ and without boundary. We will be interested in the high frequency analysis of the damped wave equation,
\begin{equation}\label{e:DWE}
\left(\partial_t^2-\Delta+2a(x)\partial_t\right)v(x,t)=0,
\end{equation}
where $\Delta$ is the Laplace-Beltrami operator on $M$ and $a\in\mathcal{C}^{\infty}(M,\mathbb{R})$ is the \emph{damping function}. The case of damping corresponds actually to $a\geq 0$ but our results will be valid for any real valued function $a$. Our main concern in this article is to study asymptotic properties of solutions of the form
$$v(t,x)=e^{-\imath t\tau}u_{\tau}(x),$$
where $\tau$ belongs to $\mathbb{C}$ and $u_{\tau}(x)$ is a non trivial element in $L^2(M)$. Such a mode is a solution of~\eqref{e:DWE} if one has
\begin{equation}\label{e:specprob}(-\Delta-\tau^2-2\imath\tau a)u_{\tau}=0.\end{equation}
From the spectral analysis of~\eqref{e:DWE}, there exist countably many $(\tau_n)$ solving this \emph{nonselfadjoint} eigenvalue problem. One can also verify that their imaginary parts remain in a bounded strip parallel to the real axis and they satisfy $\lim_{n\rightarrow+\infty}\text{Re}\ \tau_n=\pm\infty$~\cite{Sj00, Hi04, Non11}. We also recall that $(\tau,u_{\tau})$ solves the eigenvalue problem~\eqref{e:specprob} if and only if $(-\overline{\tau},\overline{u}_{\tau})$ solves it~\cite{Non11}. Our main concern in the following will be to describe some asymptotic properties of sequences $(\tau_n,u_n)_n$ solving~\eqref{e:specprob} with 
$$\text{Re}\ \tau_n\rightarrow+\infty\ \text{and}\ \text{Im}\ \tau_n\rightarrow \beta,$$
where $\beta\in\mathbb{R}$. Very general results on the asymptotic
distribution of the $\tau_n$ and its links with the properties of~\eqref{e:DWE}
 have been obtained by various authors. For instance, in a very general context, Lebeau related the geometry of the undamped geodesics, the spectral asymptotics of the $\tau_n$ and the energy decay of the damped wave equation~\cite{Leb93}. Related results were also proved in several geometric contexts where the family of undamped geodesics was in some sense not too big: closed elliptic geodesic~\cite{Hi04}, closed hyperbolic geodesic~\cite{Chr07, BuChr09}, subsets satisfying a condition of negative pressure~\cite{Sch10, Sch11, Non11}. Concerning the distribution of the $\tau_n$, Sj\"ostrand gave a precise asymptotic description of the $\tau_n$ on a general compact manifold~\cite{Sj00}. We also refer the reader to~\cite{Hi02} in the case of Zoll manifolds and to~\cite{An10a} in the case of negatively curved manifolds.

\subsection{Semiclassical reduction}
We will mention more precisely some of these results related to ours but before that we would like to proceed to a semiclassical reformulation of our problem as it was performed in~\cite{Sj00}. Thanks to the different symmetries of our problem, we can restrict ourselves to the limit $\text{Re}\ \tau\rightarrow+\infty$. We will look at eigenfrequencies $\tau$ of order $\hbar^{-1}$ (where $0<\hbar\ll 1$ will be the semiclassical parameter of our problem) and we will set
$$
\tau=\frac{\sqrt{2z}}{\hbar},\ \text{where}\
z(\hbar)=\frac{1}{2}+\mathcal{O}(\hbar).
$$
In the following, we will often omit the dependence of $z(\hbar)=z$ in $\hbar$ in order to simplify the notations. Thanks to this change of asymptotic parameters, studying the high frequency modes of the problem~\eqref{e:specprob} corresponds to look at sequences $(z(\hbar)=\frac{1}{2}+\mathcal{O}(\hbar))_{0<\hbar\ll 1}$ and $(\psi_{\hbar})_{0<\hbar\ll 1}$ in $L^2(M)$ satisfying\footnote{For simplicity of exposition, we only deal with operators of this form. However, our approach could in principle be adapted to treat the case of more general families of nonselfadjoint operators like the ones considered in~\cite{Sj00},~$\S 1$.}
\begin{equation}\label{e:eigenmode}
(\mathcal{P}(\hbar,z)-z(\hbar))\psi_{\hbar}=0,\ \text{where}\ \mathcal{P}(\hbar,z):=-\frac{\hbar^2\Delta}{2}-\imath\hbar \sqrt{2z(\hbar)} a(x).
\end{equation}
Recall that, for every $t$ in $\mathbb{R}$, the quantum propagator associated to $\mathcal{P}(\hbar,z)$ is given by
\begin{equation}\label{e:propag}
\U_{\hbar}^t:=\exp\left(-\frac{\imath
    t\mathcal{P}(\hbar,z)}{\hbar}\right).
\end{equation}

It was proved by Markus-Matsaev and Sj\"ostrand that the
``horizontal'' distribution
of the eigenvalues of $\mathcal{P}(\hbar,z)$ satisfies a Weyl law in the semiclassical limit $\hbar\rightarrow 0$ -- see Theorem $5.2$ in~\cite{Sj00} for the precise statement. Translated in this semiclassical setting, our goal is to describe asymptotic properties of a sequence of normalized eigenmodes $(\psi_{\hbar})_{\hbar\rightarrow 0^+}$ satisfying~\eqref{e:eigenmode} with
$$
z(\hbar)=\frac{1}{2}+\mathcal{O}(\hbar)\qquad \text{and}\qquad \frac{\text{Im}\
  z(\hbar)}{\hbar}= \beta+o(1),
$$
as $\hbar\rightarrow 0$. A way to study these eigenmodes is to look at the following distributions on $T^*M$~\cite{Bu97, EZ}:
\begin{equation}\label{e:distrib}
\forall b\in\mathcal{C}^{\infty}_o(T^*M),\ \mu_{\psi_{\hbar}}(b):=\langle\psi_{\hbar},\Op_{\hbar}(b)\psi_{\hbar}\rangle_{L^2(M)},
\end{equation}
where $\Op_{\hbar}(b)$ is a $\hbar$-pseudodifferential operator (see section~\ref{a:pdo} for a brief reminder). Under our assumptions, one can prove that, as $\hbar$ tends to $0$, $\mu_{\psi_{\hbar}}$ converges (up to an extraction) to a probability measure $\mu$ on the unit cotangent bundle $S^*M=\{(x,\xi)\in T^*M:\|\xi\|_x=1\}$. Moreover, this probability measure satisfies the following invariance relation:
\begin{equation}\label{e:propagationscmeas}
\forall b\in\mathcal{C}^0(S^*M),\ \mu(b)=\mu\left( b\circ
  g^te^{-2\beta t-2\int_0^ta\circ g^sds}\right),
\end{equation}
where $g^t$ is the geodesic flow on $S^*M$. Such a probability measure is called a semiclassical measure of the sequence $(\psi_{\hbar})_{\hbar\rightarrow 0^+}$~\cite{Bu97, EZ} and one can verify that the support of such a measure is invariant under the geodesic flow. Following~\cite{Leb93, Sj00, AsLeb03}, one can introduce the following dynamical quantities:
$$
A_+=\lim_{T\rightarrow+\infty}\frac{1}{T}\sup_{\rho\in
  S^*M}-\int_0^Ta\circ g^s(\rho)ds,
$$
and
$$
A_-=\lim_{T\rightarrow+\infty}\frac{1}{T}\inf_{\rho\in
  S^*M}-\int_0^Ta\circ g^s(\rho)ds.
$$
Then, $\beta\in[A_-,A_+]$. As in the selfadjoint case, one can try to understand properties of these semiclassical measures -- see~\cite{AsLeb03} for some general results. For instance, if $\{\gamma\}$ is a periodic orbit on which the Birkhoff average of $-a$,
$$\lim_{T\rightarrow+\infty}-\frac{1}{T}\int_0^Ta\circ g^s(\rho)ds,\ \rho\in\{\gamma\},$$
is not equal to $\beta$, then one has $\mu(\{\gamma\})=0$. However, if the Birkhoff average along $\gamma$ is equal to $\beta$, this can be no longer true. When specified in the case of hyperbolic periodic orbits, our main result will give informations on this kind of issues.

\subsection{Results in the selfadjoint case} Before stating our result, we would like to recall related results in the selfadjoint case $a\equiv 0$ -- see also~\cite{Ze09}, section $5$ for a more detailed account on the results we will mention. In this case, it means that we look at eigenfunctions of the Laplacian on $M$ in the large eigenvalue limit.

In~\cite{CdVPa94}, Colin de Verdi\`ere and Parisse have exhibited geometric situations where one can find a sequence of eigenmodes $(\psi_{\hbar})_{\hbar>0}$ whose semiclassical measure is an invariant probability measure carried by an hyperbolic periodic orbit $\gamma$. Yet, they show that if such a concentration occurs, it must at happen at a slow rate. Precisely, they prove that if $U$ is a fixed small neighborhood of their geodesic $\gamma$, then there exists a positive constant $C$ such that
$$\int_{M\backslash U}|\psi_{\hbar}(x)|^2d\text{vol}_M(x)\geq\frac{C}{|\log\hbar|},\ \text{as}\ \hbar\rightarrow 0.$$
This result has been generalized\footnote{As pointed out at the end of
  the appendix, our proof also allows to recover (and to generalize)
  this result.} to more general Hamiltonian flows involving a hyperbolic closed geodesic by Burq-Zworski~\cite{BuZw04} and Christianson~\cite{Chr07}. In~\cite{ToZe03}, Toth and Zelditch also consider a related question and they look at the concentration of eigenmodes in shrinking tubes in $S^*M$ of size $\hbar^{\overline{\nu}}$ around a closed hyperbolic geodesic (where $0<\overline{\nu}<\frac{1}{2}$) -- see also paragraph $5.1$ of~\cite{Ze09}. Roughly speaking, they prove that, in their specific geometric situation (completely integrable flow), not all the mass of the eigenmodes can be localized on such shrinking tubes. In this article, we will consider similar questions for more general hyperbolic subsets and for stationary modes of the damped wave equation.

Finally, under a global assumption on the geodesic flow (namely it should be Anosov), Anantharaman proved that semiclassical measures associated to eigenmodes of $\Delta$ cannot be completely carried by closed orbit of the geodesic flow (which are hyperbolic in this case)~\cite{An08}. In our main statement, \emph{we will not make any global assumption on the dynamical properties of the geodesic flow} and it would be interesting to understand how Anantharaman's statement could be extended to the damped wave equation -- see~\cite{Riv11a} for results in this sense.

\subsection{Statement of the main result} We now turn back to eigenmodes of the damped wave equation. We underline that, to the knowledge of the author, even if there is an important literature concerning eigenfunctions of the Laplacian on $M$, much less seems to be known on the asymptotic description of eigenmodes for the damped wave equation.
 Our results concerning these questions will be here of two types:
\begin{itemize}
\item we extend the study of concentration in shrinking tubes of size $\hbar^{\overline{\nu}}$ to more general hyperbolic subsets satisfying a condition of negative topological pressure;
\item we consider the situation where $a$ is a general smooth and real valued function on $M$ (and not only the case $a\equiv 0$).
\end{itemize}

As it will be involved in the statement of our main result, we recall now what is the topological pressure. Let $\Lambda$ be a compact and hyperbolic subset of $S^*M$ invariant under the geodesic flow $g^t$. For any $\epsilon>0$ and $T>0$, we say that the subset $F$ in $\Lambda$ is $(\epsilon,T)$-separated if, for any $\rho$ and $\rho'$ in $F$,
$$\forall 0\leq t\leq T,\ d(g^t\rho,g^t\rho')\leq\epsilon\ \Longrightarrow\ \rho=\rho'.$$
Then, we can define the topological pressure of the subset $\Lambda$ with respect to $\frac{1}{2}\log J^u$ where $J^u$ is the unstable Jacobian -- see paragraph~\ref{ss:anosov} below. It is defined as~\cite{Pe}
$$
P_{top}\left(\Lambda, g^t,\frac{1}{2}\log
  J^u\right):=\lim_{\epsilon\rightarrow
  0}\limsup_{T\rightarrow+\infty}\frac{1}{T}\log\sup_{F}\left\{\sum_{\rho\in
    F}\exp\left(\frac{1}{2}\int_0^T\log J^u\circ
    g^s(\rho)ds\right)\right\},
$$
where the supremum is taken over all $(\epsilon,T)$-separated subsets
$F$. In this definition, we have two phenomena. On the one hand, the
Birkhoff average of $\frac{1}{2}\log J^u$ leads to exponentially small
terms when $T\to\infty$; on the other hand, depending on the
complexity of the dynamics on $\Lambda$, the cardinal of $F$ could
grow exponentially when $T\to\infty$. Thus, saying that the topological pressure is negative means that the contribution of the first quantity is more important. If $\Lambda$ is a (or a collection of) closed hyperbolic geodesics, then $P_{top}\left(\Lambda, g^t,\frac{1}{2}\log J^u\right)$ is negative.

We say that a function is $(\Lambda,\hbar,\overline{\nu})$- localized if it is a smooth cutoff function in a $\hbar^{\overline{\nu}}$-neighborhood of $\Lambda$ -- see $\S$~\ref{ss:localize} for a precise definition. We can now state our main result.

\begin{theo}\label{t:maintheo1} Suppose $\Lambda$ is a compact, invariant, hyperbolic subset satisfying 
$$P_{top}\left( \Lambda,g^t,\frac{1}{2}\log J^u\right)<0,$$ 
and such that
\begin{equation}\label{e:ptylambda}
\sup_{\rho\in\Lambda}-\int_0^Ta\circ g^s(\rho) ds\leq\beta
T+\mathcal{O}(1)\quad \text{when}\ T\rightarrow+\infty\,.
\end{equation}

Fix $0<\overline{\nu}<\frac{1}{2}$ and a $(\Lambda,\hbar,\overline{\nu})$\emph{-localized} function $\Theta_{\Lambda,\hbar,\overline{\nu}}$. 

Then, there exists a constant $c_{\Lambda,a,\overline{\nu}}<1$ such that, for any sequence $(\psi_{\hbar})_{\hbar\rightarrow 0^+}$ of eigenmodes satisfying~(\ref{e:eigenmode}) with
$$z(\hbar)=\frac{1}{2}+\mathcal{O}(\hbar)\quad \text{and}\quad \frac{\Im\ z(\hbar)}{\hbar}\geq\beta+o\left(|\log\hbar|^{-1}\right),\quad \text{as}\ \hbar\rightarrow 0^+,$$
one has
$$\limsup_{\hbar\rightarrow0}\left\langle\Op_{\hbar}\left(\Theta_{\Lambda,\hbar,\overline{\nu}}\right) \psi_{\hbar},\psi_{\hbar}\right\rangle\leq c_{\Lambda,a,\overline{\nu}}<1.$$
\end{theo}

We underline that we allow the imaginary parts of $z(\hbar)$ to go a
little bit below the horizontal axis $\{\Im
z=\hbar\beta\}$. Precisely, we authorize an error of order
$o(\hbar|\log\hbar|^{-1})$, that will be crucial for the results
proven in the appendix. A more comfortable statement is given by the
following corollary which can be deduced from
Theorem~\ref{t:maintheo1}:

\begin{coro}\label{t:maincro}
 Suppose $\Lambda$ is a compact, $(g^t)_t$-invariant hyperbolic satisfying 
$$P_{top}\left( \Lambda,g^t,\frac{1}{2}\log J^u\right)<0.$$ 
 Suppose also that there exists a positive constant $C$ such that
$$\forall T>0,\ \forall\rho\in\Lambda,\qquad -C+\beta T\leq-\int_0^Ta\circ g^s(\rho) ds\leq\beta T+C.$$

 Fix $0<\overline{\nu}<\frac{1}{2}$ and a $(\Lambda,\hbar,\overline{\nu})$\emph{-localized} function $\Theta_{\Lambda,\hbar,\overline{\nu}}$. 

Then, there exists a constant $c_{\Lambda,a,\overline{\nu}}<1$ such that, for any sequence $(\psi_{\hbar})_{\hbar\rightarrow 0^+}$ of eigenmodes satisfying~(\ref{e:eigenmode}) with $z(\hbar)=\frac{1}{2}+\mathcal{O}(\hbar)$ as $\hbar\rightarrow 0^+$,
one has
$$\limsup_{\hbar\rightarrow0}\left\langle\Op_{\hbar}\left(\Theta_{\Lambda,\hbar,\overline{\nu}}\right) \psi_{\hbar},\psi_{\hbar}\right\rangle\leq C_{\Lambda,a,\overline{\nu}}<1.$$
\end{coro}

\begin{proof}
Let us briefly explain how Corollary~\ref{t:maincro} can be obtained from Theorem~\ref{t:maintheo1}. One can proceed by contradiction and suppose that there exists a sequence $(\hbar_l\searrow 0)_{l\in\mathbb{N}}$ and a sequence $(\psi_{\hbar_l})_l$ of normalized eigenmodes satisfying~(\ref{e:eigenmode}) with $z(\hbar_l)=\frac{1}{2}+\mathcal{O}(\hbar_l)$ and
$$\lim_{l\rightarrow+\infty}\left\langle\Op_{\hbar_l}\left(\Theta_{\Lambda,\hbar_l,\overline{\nu}}\right) \psi_{\hbar_l},\psi_{\hbar_l}\right\rangle=1.$$
This implies that, for any semiclassical measure $\mu$ associated to
this sequence, one must have $\mu(\Lambda)=1$. In particular, thanks
to relation~\eqref{e:propagationscmeas}, this implies that $\frac{\Im\
  z(\hbar_l)}{\hbar_l}$ tends to $\beta$ as $l$ tends to
infinity. Thanks to Theorem~\ref{t:maintheo1}, one also obtains that
$\frac{\Im\ z(\hbar_{l'})}{\hbar_{l'}}\leq\beta$ for an infinite
subsequence of integers $l'$. On the other hand, our assumption also implies that
\begin{equation}\label{e:contradcoro}\lim_{l'\rightarrow+\infty}\left\langle\Op_{\hbar_{l'}}\left(\tilde{\Theta}_{\Lambda,\hbar_{l'},\overline{\nu}}\right) \overline{\psi_{\hbar_{l'}}},\overline{\psi_{\hbar_{l'}}}\right\rangle=1,
\end{equation}
where
$\tilde{\Theta}_{\Lambda,\hbar_{l'},\overline{\nu}}(x,\xi)=\Theta_{\Lambda,\hbar_{l'},\overline{\nu}}(x,-\xi).$
The function $\tilde{\Theta}_{\Lambda,\hbar_{l'},\overline{\nu}}$
satisfies the assumption of Theorem~\ref{t:maintheo1} with the set
$\Lambda$ replaced by
$\Lambda':=\{(x,\xi):(x,-\xi)\in\Lambda\}$. Moreover, the sequence $(
\overline{\psi_{\hbar_{l'}}})_{l'}$ solves~(\ref{e:eigenmode}) if we
replace $a$ by $-a$ and $z(\hbar_{l'})$ by
$\overline{z(\hbar_{l'})}$. In particular, since $\frac{\Im\overline{\
    z(\hbar_{l'})}}{\hbar_{l'}}\geq-\beta$ and since $\Lambda'$ satisfies
the assumption of Theorem~\ref{t:maintheo1} w.r.t. the pair $(-a,-\beta)$, we can
apply the Theorem to this new sequence: the conclusion of the Theorem contradicts the limit~\eqref{e:contradcoro}.

\end{proof}

In the selfadjoint case $a\equiv 0$, this corollary slightly improves
Toth-Zelditch's result as we only impose the hyperbolic subsets to
satisfy a condition of negative topological pressure. A default of our
approach is yet that the upper bound $c_{\Lambda,a,\overline{\nu}}$ is
not very explicit compared to the constant appearing in~\cite{Ze09} --
section $5$. Our interest in proving this result was also to show that
this property remains true in the nonselfadjoint case where $a$ is non
constant. As was already mentioned, nothing forbids a priori that
eigenmodes with damping parameter $\beta$ concentrate on a
$\beta$-damped closed geodesic\footnote{In the selfadjoint case
  ($a\equiv 0,\ \beta=0$), Colin de Verdi\`ere $\&$ Parisse's example satisfies such a property.}: corollary~\ref{t:maincro} prevents fast concentration on such orbits if they are hyperbolic.

If the geodesic flow is ergodic for the Liouville measure on $S^*M$ (manifolds of negative curvature are the main example), Sj\"ostrand showed that most of the imaginary parts converge to the spatial average of $-a$~\cite{Sj00}. Thus, in this case, our result says that if there is a hyperbolic closed geodesic with such a Birkhoff average, then eigenmodes cannot concentrate on it too fast. As was already pointed out, it would be interesting to understand what can be said under the additional assumption that the geodesic flow is Anosov on $S^*M$ (e.g. if $M$ is of negative curvature). For instance, can one prove in the Anosov case that semiclassical measures cannot be completely carried by a $\beta$-damped closed orbit?

Finally, we would like to say a few words about the proof. Our argument relies crucially on hyperbolic dispersive estimates as they were obtained by Anantharaman and Nonnenmacher in the Anosov case~\cite{An08, AN07} and by Nonnenmacher and Zworski in the context of chaotic scattering~\cite{NZ09}. More precisely, we will use a generalization of these properties in a nonselfadjoint setting similar to the results obtained by Schenck in~\cite{Sch10}.

These hyperbolic estimates give an upper bound for the growth of ``quantum cylinders'' associated to $\psi_{\hbar}$ and localized near the hyperbolic set $\Lambda$. These cylinders are a kind of analogues in a quantum setting of the Bowen balls used in the theory of dynamical systems~\cite{KaHa, Pe}. Under our dynamical assumption on $\Lambda$, one can show that the mass of ``quantum cylinders'' near the set $\Lambda$ is exponentially small for cylinders of length $\mathcal{K}|\log\hbar|$ (with $\mathcal{K}>0$ very large but independent of $\hbar$) -- paragraph~\ref{pa:hypdispest}. Then, the main difficulty is that it is hard to connect these estimates for long cylinders to estimates which are valid for shorter cylinders to which we could apply the semiclassical approximation, e.g. of length less than the Ehrenfest time $\kappa_0|\log\hbar|$~\cite{BoRo02} (with $\kappa_0>0$ small independent of $\hbar$). It turns out that if we restrict ourselves to cylinders that remain in a $\hbar^{\overline{\nu}}$-
neighborhood of $\Lambda$, the mass on the quantum cylinders (far from this neighborhood) is positive and it satisfies a ``subadditive structure'' -- paragraph~\ref{pa:submult}. A similar property was already observed and used by Anantharaman in a selfadjoint context~\cite{An08}. In our case, it implies that if the mass on the cylinders of length $\mathcal{K}|\log\hbar|$ far from the $\hbar^{\overline{\nu}}$-neighborhood is positive, then this property remains true for cylinders of shorter length $\kappa_0|\log\hbar|$. This observation is crucial in our proof and it allows to get the conclusion using standard semiclassical rules-- paragraph~\ref{pa:semiclapprox}

\subsection*{Organization of the article} In section~\ref{s:dynamical}, we introduce the dynamical setting of the article. We also build an open cover of $S^*M$ that will be used to define quantum cylinders in the subsequent section. Then, in section~\ref{s:proof}, we give the proof of Theorem~\ref{t:maintheo1} and postpone the proof of several semiclassical results to section~\ref{s:largesums}. In section~\ref{a:pdo}, we give a short toolbox on pseudodifferential calculus on a manifold. 

Finally, in an appendix in collaboration with St\'ephane Nonnenmacher, we explain how these methods can be used to derive inverse logarithmic spectral gaps for the damped wave equation -- see~\cite{Chr07, BuChr09, CSVW12} for related results.

\section{Dynamical setting}\label{s:dynamical}

The Hamiltonian function associated to the geodesic flow on $S^*M$ will be denoted $p_0(x,\xi)=\frac{\|\xi\|_x^2}{2}$ in the following of this article. Under proper assumptions (see remark~\ref{r:generalization}), we underline that our proof should also work for more general Hamiltonian flows as in~\cite{NZ09, Sj00}; yet, for simplicity of exposition, we restrict ourselves to the case of geodesic flows.   

\subsection{Hyperbolic sets}\label{ss:anosov}

From this point, we make the assumption that the set $\Lambda$ is a compact, invariant and hyperbolic subset of $S^*M$. The hyperbolicity hypothesis means that one has the following decomposition~\cite{KaHa}

$$\forall\rho\in\Lambda,\ T_{\rho}S^*M=\mathbb{R}X_{p_0}(\rho)\oplus E^u(\rho)\oplus E^s(\rho),$$
where $\mathbb{R}X_{p_0}(\rho)$ is the direction of the Hamiltonian vector field, $E^u(\rho)$ is the unstable space and $E^s(\rho)$ is the stable space. In particular, there exist a constant $C>0$ and $0<\lambda<1$ such that for every $t\geq 0$, one has
$$\forall v^u\in E^u(\rho),\ \|d_{\rho}g^{-t}v^u\|\leq C\lambda^{t}\|v^u\|\ \text{and}\ \forall v^s\in E^s(\rho),\ \|d_{\rho}g^{t}v^s\|\leq C\lambda^{t}\|v^s\|.$$
Due to the specific structure of our Hamiltonian, the above properties remain true for any energy layer\footnote{For more general Hamiltonian, it would remain true in a small vicinity of the energy layer due to the stability of the hyperbolic structure~\cite{KaHa}.} associated to $E>0$
$$\mathcal{E}_E:=p_0^{-1}\left(\{E\}\right)=\left\{(x,\xi)\in T^*M:p_0(x,\xi)=E\right\}.$$
Define now the unstable Jacobian at point $\rho\in S^*M$ and time $t\geq 0$
$$J^u_t(\rho):=\left|\det\left(d_{g^t\rho}g^{-t}_{|E^u(g^t\rho)}\right)\right|,$$
where the unstable spaces at $\rho$ and $g^t\rho$ are equipped with the induced Riemannian metric. It defines a H\"older continuous function on $S^*M$~\cite{KaHa} (that can be extended to any energy layer $\mathcal{E}_E$). We underline that this quantity tends to $0$ with an exponential rate as $t$ tends to infinity. Moreover, it satisfies the following multiplicative property
$$J^u_{t+t'}(\rho)=J^u_t(g^{t'}\rho)J^u_{t'}(\rho).$$
In the following, we will use the notation $J^u(\rho)=J^u_{1}(\rho)$ on $S^*M$.

\subsection{Topological pressure}\label{pa:toppress} 

In the statement of Theorem~\ref{t:maintheo1}, we made an assumption on the topological pressure of the subset $\Lambda$. Let us explain what informations are provided by this hypothesis following the observations of paragraph $5.2$ in~\cite{NZ09} -- see also~\cite{Pe}, chapter $4$ for general definitions of topological pressure. 

Fix a small $\delta>0$. Then, for every $E\in [\frac{1-\delta}{2},\frac{1+\delta}{2}]$, the set
$$\Lambda_{E}=\left\{(x,\xi)\in \mathcal{E}_E:\left(x,\frac{\xi}{\sqrt{2E}}\right)\in \Lambda\right\}$$
is hyperbolic. We fix a finite open cover $\mathcal{V}=(V_a)_{a\in A}$ of 
\begin{equation}\label{e:Lambda^delta}
\Lambda^{\delta}:=\bigcup_{\frac{1-\delta}{2}\leq E\leq
  \frac{1+\delta}{2}}\Lambda_E
\end{equation}
of diameter less than some small $\epsilon>0$ and such that, for every $a$ in $A$, one has
$$V_a\subset\cE^{\delta}:= p_0^{-1}\big(
(1/2 - \delta,1/2 +\delta) \big)\,. $$

For every integer $n_0$, the refined cover $\mathcal{V}^{(n_0)}$ is
the collection of the open sets
$$V_{\alpha}=\bigcap_{j=0}^{n_0-1}g^{-j}V_{\alpha_j},\ \text{where}\ \alpha=(\alpha_0,\alpha_1,\ldots,\alpha_{n_0-1})\in A^{n_0}.$$
Equivalently, $V_{\alpha}$ contains the points $\rho$, the trajectory
of which sits in $V_{\alpha_0}$ at time $0$, in $V_{\alpha_1}$ at time $1$, etc, and
in $V_{\alpha_{n-1}}$ at time $n-1$.

The fact that $P_{top}(\Lambda,g^t,\frac{1}{2}\log J^u)<0$ implies the
existence of a positive constant $P_0$ such that for $\delta$ small enough, for any cover of
small enough diameter (say $\epsilon\leq\epsilon_0$) and for any
$n_0\in\mathbb{N}$ large enough (depending on $\epsilon$), one can
extract a subcover $\mathcal{W}^{(n_0)}\subset\mathcal{V}^{(n_0)}$ of
$\Lambda^\delta$ such that
\begin{equation}\label{e:Tpressure}\sum_{V_{\alpha}\in \mathcal{W}^{(n_0)}}\sup_{\rho\in V_{\alpha}\cap\Lambda^{\delta}}\left\{\exp\left(\frac{1}{2}\int_0^{n_0}\log J^u\circ g^t(\rho)dt\right)\right\}\leq e^{-2n_0P_0}\end{equation}
(we may assume that any $V_{\alpha}\in \mathcal{W}^{(n_0)}$ intersects
$\Lambda^\delta$). Thanks to assumption~\eqref{e:ptylambda} on $\Lambda$, we can also verify that for $n_0$ large enough, one also has
\begin{equation}\label{e:assumpress}\sum_{V_{\alpha}\in \mathcal{W}^{(n_0)}}\sup_{\rho\in V_{\alpha}\cap\Lambda^{\delta}}\left\{\exp\left(\int_0^{n_0}\left(\frac{1}{2}\log J^u-a\right)\circ g^t(\rho)dt\right)\right\}\leq e^{n_0(\beta-P_0)}.\end{equation}

\begin{rema}\label{r:diameter} In our proof, we will fix an open cover of small diameter $\epsilon\leq\epsilon_0$ in order to get a subcover $\mathcal{W}^{(n_0)}$ satisfying~\eqref{e:assumpress}. Such a choice can be made for every $\epsilon\leq\epsilon_0$. Moreover, we choose such an epsilon in order to have $\epsilon\leq\tilde{\epsilon}_0/2$, where $\tilde{\epsilon}_0$ is the constant appearing in lemma~\ref{l:BowenRuelle}. We also take $\epsilon$ small enough to have the factor $1+\mathcal{O}(\epsilon)$ in estimate~\eqref{e:HDE0} smaller than $e^{\frac{P_0}{2}}$.
\end{rema}

Once $\mathcal{V}$ is chosen with the above requirements, we also select $n_0$ and
$\mathcal{W}^{(n_0)}$ such that \eqref{e:assumpress} holds. All these parameters will remain fixed for the rest of the proof.\\

We will call $W$ the
family of words $\alpha=(\alpha_0,\alpha_1,\ldots,\alpha_{n_0-1})$
corresponding to the elements $V_\alpha\in
\mathcal{W}^{(n_0)}$. We also complete the cover, by selecting an open set $V_{\infty}$ such
that $\overline{V_{\infty}}\cap\Lambda^{\delta}=\emptyset$, and such that
$$V_{\infty}\cup \left(\bigcup_{\alpha\in W}V_{\alpha}\right)=\cE^{\delta}.$$
Finally, we denote $\overline{W}=W\cup\{\infty\}$.

\subsection{A lemma from dynamical systems} Before entering the details of our proof, we mention the following lemma which is taken from the appendix of~\cite{BoRue75} (lemma $A.2$):

\begin{lemm}\label{l:BowenRuelle}  Let $\Lambda$ be a hyperbolic set in
  $S^*M$ satisfying assumption~\eqref{e:ptylambda}. There exists $\tilde{\epsilon}_0>0$ (depending on $M$,
 $\delta$ and $a(x)$) such that, for any $E\in [\frac12-\delta,\frac12+\delta]$, for any $p>0$ and any $\rho_2\in \cE^{\delta}$ satisfying
$$\exists \rho_1\in\Lambda^{\delta}\ \text{such that}\ \forall 0\leq k\leq p-1,\quad d_{T^*M}(g^k\rho_1,g^k\rho_2)\leq\tilde{\epsilon}_0,$$
one has
$$-\int_0^pa\circ g^s(\rho_2)ds\leq \beta p+\mathcal{O}(1),$$
where the constant involved in $\mathcal{O}(1)$ is independent of $\rho_2$ and $p$.
\end{lemm}

In particular, this lemma will allow us to extend the inequality \eqref{e:ptylambda} to a small (dynamical) neighborhood of
$\Lambda^{\delta}$. The proof of this lemma was given
in~\cite{BoRue75} where the authors treated the case of a single energy layer ($\delta=0$). Yet, their proof can be adapted to get a uniform $\tilde{\epsilon}_0$ on the energy interval $\cE^{\delta}$. We verify below that their argument can be extended to a small neighborhood of $S^*M$.

\begin{proof}  
The proof of this lemma relies on two observations:
\begin{itemize}
\item if the trajectory of $\rho_2$ remains close to the one of $\rho_1\in\Lambda^{\delta}$ in the future, then $\rho_2$ must have an ``exponentially small unstable component'';
\item the Birkhoff averages $-\int_0^pa\circ g^sds$ on $\Lambda^{\delta}$ are uniformly bounded by $\beta p+\mathcal{O}(1)$.
\end{itemize}

We closely follow the presentation of~\cite{BoRue75} and refer the reader to it for more details. We start by giving a precise meaning to the first observation. For that purpose, we write the following decomposition of the tangent space, for any $\rho=(x,\xi)\in\Lambda^{\delta}$,
$$T_{\rho}\cE^{\delta}=E^0(\rho)\oplus E^s(\rho)\oplus E^u(\rho),$$
where $E^0(\rho)$ is the vector space generated by $X_{p_0}(\rho)$ and the energy direction $\rho(t)=(x,t\xi)$ and $E^{u/s}$ are still the unstable/stable directions. 
For $v$ in $T_{\rho}\cE^{\delta}$, we denote $v=v_0+v_s+v_u$ the decomposition adapted to these subspaces. For $\epsilon'>0$ small enough and
any $\rho\in\Lambda^{\delta}$, one can construct a smooth
chart\footnote{Here, $H(\epsilon')$ means that we consider a ball of
  radius $\epsilon'$ around $0$ in the subspace $H$.}
$\phi_{\rho}:T_{\rho}\cE^{\delta}(\epsilon')\rightarrow \cE^{\delta}$
satisfying
$$\phi_{\rho}\left[(E^0(\rho)+E^s(\rho))(\epsilon')\right]\subset W^{cs}(\rho),\ \text{and}\ \phi_{\rho}\left[(E^0(\rho)+E^u(\rho))(\epsilon')\right]\subset W^{cu}(\rho),$$
with
$$W^{cs/cu}(x,\xi)=\bigcup_{t\in\mathbb{R}}\bigcup_{\frac{1-\delta}{2}\leq E\leq\frac{1+\delta}{2}} W^{s/u}\left(g^t\left(x,\sqrt{2E}\frac{\xi}{\|\xi\|}\right)\right),$$
where $W^{s/u}(\rho')$ denote the stable/unstable manifold at point $\rho'$. Moroeover, one can choose $\phi_{\rho}$
such that $d_0\phi_{\rho}$ is given by the identity. The construction
is a straightforward adaptation of property $A.1$ in~\cite{BoRue75} to
a small neighborhood of $S^*M$.

For $\epsilon'>0$ small enough, introduce now
$$F_{\rho}=\phi_{g^1\rho}^{-1}\circ g^1\circ \phi_{\rho}:T_{\rho}\cE^{\delta}(\epsilon')\rightarrow T_{g^1\rho}\cE^{\delta},$$
which is tangent to $d_{\rho}g^1$ at the origin. Define also 
$$D(\epsilon', p):=\{v\in T_{\rho}\cE^{\delta}:\forall0\leq k\leq p-1,\ \|F_{g^k\rho}\circ\ldots F_{\rho}v\|_{g^k\rho}\leq\epsilon'\}.$$
Let $v=v_0+v_s+v_u$ be an element in $D(\epsilon', p)$. One can mimick again the proof of~\cite{BoRue75} (precisely the proof of inequality $A.5$ in this reference) and verify that there exist uniform constants $C>0$ and $0<\lambda<1$ such that
\begin{equation}\label{e:borubound}\forall 0\leq k\leq p-1,\ \|F_{g^k\rho}\circ\ldots F_{\rho}(v_0+v_s+v_u)-F_{g^k\rho}\circ\ldots F_{\rho}(v_0+v_s)\|_{g^{k}\rho}\leq C\epsilon'\lambda^{p-1-k}.\end{equation}
This upper bound is obtained thanks to the hyperbolicity assumption (combined to a Taylor formula near the origin). This result expresses the first property mentionned at the beginning of our proof. Precisely, it shows that a point which remains close to $\rho_1\in\Lambda^{\delta}$ during a time $p$ has an exponentially small unstable component (in our system of charts).\\

We will now use this family of charts to prove lemma~\ref{l:BowenRuelle}. First, we observe that there exists a constant $C>0$ such that
$$\forall \rho_1,\rho_2\in\cE^{\delta},\  \left|\int_0^1a\circ g^s(\rho_1)ds-\int_0^1a\circ g^s(\rho_2)ds\right|\leq Cd(\rho_1,\rho_2).$$
Fix now $\rho_1=(x_1,\xi_1)$ in $\Lambda^{\delta}$ and $\rho_2$ in $\cE^{\delta}$ satisfying
$$\forall 0\leq k\leq p-1,\quad d_{T^*M}(g^k\rho_1,g^k\rho_2)\leq\tilde{\epsilon}_0,$$
where $\tilde{\epsilon}_0$ is some small positive parameter. In
particular, we choose it small enough to have
$v=\phi_{\rho_1}^{-1}(\rho_2)$ belongs to
$T_{\rho}\cE^{\delta}(\epsilon'/2)$ for every
$\rho_1\in\Lambda^{\delta}$ and any $\rho_2\in\cE^{\delta}$ satisfying
$d(\rho_1,\rho_2)\leq\tilde{\epsilon}_0$. Define then $w=v_0+v_s$ and
introduce $\rho_3=\phi_{\rho_1}(w)$. Thanks to our construction, one
has $\rho_3\in W^s(g^{\tau}\tilde{\rho}_1)$ for some $|\tau|\leq
C_0\epsilon'$ and some $\tilde{\rho}_1=(x_1,E'\xi_1)\in \Lambda_{E'}$ with $|E'|\leq
C_0\epsilon'$. Thanks to the fact that $a$ does not depend on $\xi$
and that $\tilde{\rho}_1$ belongs to $\Lambda_{E'}$, the
assumption~\eqref{e:ptylambda} directly implies that
$$-\int_0^pa\circ g^s(\tilde{\rho}_1)ds\leq\beta p+\mathcal{O}(1),$$
where the constant involved in the remainder is uniform for $|E'|\leq
C_0\epsilon'$. To extend the assumption~\eqref{e:ptylambda} to every
energy layer $\Lambda_{E'}$, we have crucially used the fact that $a$
is independent of $\xi$, and the homogeneity of the geodesic flow --- see remark~\ref{r:generalization} below for generalizations of this fact.

We will now compare the average along the trajectory of $\rho_2$ with the average along the trajectory of $\tilde{\rho}_1$. Thanks to the upper bound~\eqref{e:borubound} and to the construction of $\rho_3$, one has that, for any $0\leq k\leq p-1$, 
$$d(g^k\rho_3,g^k\rho_2)\leq C_1\epsilon'\lambda^{p-1-k}\quad \text{and}\quad d(g^{k+\tau}\tilde{\rho}_1,g^k\rho_3)\leq C_1'\epsilon'(\lambda')^{k},$$ 
for some uniform $C_1'>0$ and $0<\lambda'<1$. We now use these properties to bound $-\int_0^pa\circ g^s(\rho_2)ds$. We write
$$-\int_0^pa\circ g^s(\rho_2)ds\leq-\int_0^pa\circ g^s(\tilde{\rho}_1)ds
+\left|\int_0^pa\circ g^s(\tilde{\rho}_1)ds-\int_0^pa\circ g^{s+\tau}(\tilde{\rho}_1)ds\right|$$
$$+\sum_{k=0}^{p-1}\left(\left|\int_0^1a\circ g^{s+k+\tau}(\tilde{\rho}_1)ds-\int_0^1a\circ g^{s+k}(\rho_3)ds\right|
+\left|\int_0^1a\circ g^{s+k}(\rho_3)ds-\int_0^1a\circ g^{s+k}(\rho_2)ds\right|\right) 
.$$
Using the different properties mentioned above, one gets 
$$-\int_0^pa\circ g^{s}(\rho_2)ds\leq\beta p+\mathcal{O}(1)+2C_0\epsilon'\|a\|_{\infty}+\frac{CC_1'\epsilon'}{1-\lambda'}+\frac{CC_1\epsilon'}{1-\lambda},$$
which is the expected conclusion.
\end{proof}

\begin{rema}\label{r:generalization} 
At this point, we would like to mention something on the generalization of Theorem~\ref{t:maintheo1} to more general nonselfadjoint operators as in~\cite{Sj00}. In order to adapt the previous lemma (which will be crucial in our proof) for more general Hamiltonian flows, one has to make the assumption that the Birkhoff averages of the corresponding damping function are bounded by $\beta p+\mathcal{O}(1)$ for every trajectory in a small neighborhood $\Lambda^{\delta}$ of the hyperbolic subset $\Lambda$. Here this property was satisfied due to the specific structure of the ``damping function'' $a$ and of the geodesic flow.
 
\end{rema}

\section{Proof of the main Theorem}\label{s:proof}

We fix $\beta$ a spectral parameter. Let $(\psi_{\hbar})_{0<\hbar\leq \hbar_0}$ be a sequence of normalized vector in $L^2(M)$ such that
$$\mathcal{P}(\hbar,z)\psi_{\hbar}=z(\hbar)\psi_{\hbar},$$
where $z(\hbar)$ satisfies
\begin{equation}
 \label{e:spectralpara}
\forall 0<\hbar\leq \hbar_0,\ z(\hbar)=\frac{1}{2}+\mathcal{O}(\hbar)\ \text{and}\ \text{Im}\ z(\hbar)\geq\beta\hbar+o\left(\hbar|\log\hbar|^{-1}\right).
\end{equation}

\begin{rema}
 Such a family may be defined by a discrete sequence $\hbar_n\rightarrow 0$ as $n$ tends to infinity. Yet, in order to avoid heavy notations and to fit semiclassical notations~\cite{DS, EZ}, we will use the standard convention $\hbar\rightarrow 0$ to denote the limit.
\end{rema}

\subsection{Concentration properties and discretization of the energy layer}

In this paragraph, we describe the setting we will use to prove
Theorem~\ref{t:maintheo1}. We introduce $\Lambda$ a compact,
hyperbolic and invariant subset of $S^*M$ satisfying the assumption \eqref{e:ptylambda}.
As in paragraph~\ref{pa:toppress}, we fix a small neighborhood of size
$\delta>0$ around $S^*M$ (thanks to our assumption on $\text{Re}\
z(\hbar)$, the eigenmodes are microlocalized on $S^*M$ when $\hbar$ tends to $0$).

We make the assumption that $P_{top}\left( \Lambda,g^t,\frac{1}{2}\log J^u\right)<0$ and we will use the open covers introduced in~$\S \ref{pa:toppress}$.

\subsubsection{Cutoff functions near $\Lambda$}\label{ss:localize}

We fix $0<\overline{\nu}<1/2$ a positive parameter and we introduce a cutoff function $0\leq \Theta_{\Lambda,\hbar,\overline{\nu}}\leq 1$ around the set $\Lambda$. This function belongs to $\mathcal{C}^{\infty}_c(T^*M)$ and satifies the following assumptions:
\begin{itemize}
 \item $\Theta_{\Lambda,\hbar,\overline{\nu}}(x,\xi)=0$ for $\|\xi\|^2\notin[1/4,2]$;
 \item $\Theta_{\Lambda,\hbar,\overline{\nu}}(x,\xi)=\Theta_{\Lambda,\hbar,\overline{\nu}}(x,\xi/\|\xi\|)$ for $\|\xi\|^2\in[1/2,3/2]$;
 \item for every $\rho$ in $S^*M$ satisfying $d(\rho,\Lambda)\leq\hbar^{\overline{\nu}}/2$, $\Theta_{\Lambda,\hbar,\overline{\nu}}(\rho)=1$;
 \item for every $\rho$ in $S^*M$ satisfying $d(\rho,\Lambda)\geq2 \hbar^{\overline{\nu}}$, $\Theta_{\Lambda,\hbar,\overline{\nu}}(\rho)=0$;
 \item the growth of the derivatives of  $\Theta_{\Lambda,\hbar,\overline{\nu}}$ is controlled by powers of $\hbar^{-\overline{\nu}}$ and so the functions are amenable to $\hbar$-pseudodifferential calculus~\cite{DS, EZ} (see also appendix~\ref{a:pdo} for a brief reminder);
\end{itemize}

We say that such a function is $(\Lambda,\hbar,\overline{\nu})$\emph{-localized}. Our goal is to prove that
\begin{equation}\label{e:contradictionlim}
\liminf_{\hbar\rightarrow0}\left\langle\Op_{\hbar}\left(1- \Theta_{\Lambda,\hbar,\overline{\nu}}\right) \psi_{\hbar},\psi_{\hbar}\right\rangle\geq c_{\Lambda,a,\overline{\nu}}>0,
\end{equation}
for some positive constant $c_{\Lambda,a,\overline{\nu}}$ that depends only on $\Lambda$, $a$ and $\overline{\nu}$ (and, in particular, not on the sequence $(\psi_{\hbar})_{\hbar\rightarrow 0}$).

\subsubsection{Smooth discretization of the energy layer}\label{s:discretization}

We now introduce a smooth partition of unity associated to our open
cover $(V_{\alpha})_{\alpha\in\overline{W}}$, namely a family of
smooth functions $P_\alpha\in C^\infty_c(V_\alpha,[0,1])$ which
satisfy
$$
\sum_{\alpha\in \overline{W}}P_{\alpha}(\rho)=1\quad \text{near}\ \mathcal{E}^{\delta/2}.
$$

This smooth partition can be quantized into a family of
pseudodifferential operators
$(\pi_{\alpha}\in\Psi^{-\infty,0}(M))_{\alpha\in\overline{W}}$ such
that for each $\alpha\in\overline{W}$, $P_\alpha$ is the principal symbol of $\pi_\alpha$, and
$$
WF_{\hbar}(\pi_{\alpha})\subset V_{\alpha},\quad
\pi_{\alpha}^*=\pi_{\alpha}\quad \text{and}\quad
\sum_{\alpha}\pi_{\alpha}=\text{Id}\ \text{microlocally near}\
\mathcal{E}^{\delta/2},
$$

We also introduce the following ``refined'' operators:
$$\forall\gamma=(\gamma^0,\gamma^1,\ldots,\gamma^{n-1})\in
\overline{W}^n,\quad 
\Pi_{\gamma}:=\U_{\hbar}^{n_0}\pi_{\gamma^{n-1}}\U_{\hbar}^{n_0}\ldots\pi_{\gamma^1}\U_{\hbar}^{n_0}\pi_{\gamma^0},\quad 
\tPi_{\gamma}:= \Pi_{\gamma} \U_{\hbar}^{-nn_0}\,.
$$
This new family of operators satisfies
\begin{equation}\label{e:nonsapartition}
 \sum_{|\gamma|=n}\Pi_{\gamma}=\U_{\hbar}^{nn_0}\ \text{microlocally near}\ \mathcal{E}^{\delta/2}\,,
\end{equation}
equivalently 
$$
 \sum_{|\gamma|=n}\tPi_{\gamma}=\text{Id}\ \text{microlocally near}\ \mathcal{E}^{\delta/2}\,,
$$
uniformly for times $0\leq n\leq C|\log\hbar|$, for any fixed $C>0$. 

We notice that for $n=|\gamma|$ finite, each operator $\tPi_\gamma$
admits for principal symbol 
\begin{equation}\label{e:tP}
\tP_{\gamma}:= P_{\gamma^{n-1}}\circ g^{-n_0}\ldots P_{\gamma^1}\circ
  g^{(1-n)n_0}P_{\gamma^0}\circ g^{-nn_0}\,,
\end{equation}
which is supported in the ``backward refined
set''\footnote{$\tV_{\gamma}$ contains the points $\rho$ which were sitting in
  $V_{\gamma^{n-1}}$ at time $-n_0$, in $V_{\gamma^{n-2}}$ at time
  $-2n_0$,..., in $V_{\gamma^{0}}$ at time
  $-nn_0$. The word $\gamma$ thus describes the backward trajectory of
$\rho$.}
$$
\tV_{\gamma}:=g^{n_0}V_{\gamma^{n-1}} \cap g^{2n_0}V_{\gamma^{n-2}}\cap\cdots \cap g^{nn_0}V_{\gamma^0}\,.
$$
In subsection~\ref{pa:largesums} we will see that this connection between
$\tPi_\gamma$ and $\tP_\gamma$ extends to times $n\leq
\kappa_0|\log\hbar|$, for $\kappa_0>0$ small enough. 

We already have two families of $n$-cylinders: the full set of $n$-cylinders
$$\overline{W}^n= \left\{(\gamma^0,\gamma^1,\ldots,\gamma^{n-1}):\ \forall 0\leq j\leq n-1,\ \gamma^j\in\overline{W}\right\},$$
covering the whole energy slab $\cE^\delta$, 
and the set of $n$-cylinders
$$W^n=\left\{(\gamma^0,\gamma^1,\ldots,\gamma^{n-1}):\ \forall
  0\leq j\leq n-1,\ \gamma^j\in W\right\},$$
corresponding to trajectories remaining $\eps$-close to
$\Lambda^\delta$ during a time $nn_0$.

We will distinguish a subfamily of $n$-cylinders, corresponding to points
\emph{very close to} $\Lambda$. Namely, we define
$\Lambda_n\subset \overline{W}^n$ to be the set of $n$-cylinders satisfying
$$
\text{supp}\big(\Theta_{\Lambda,\hbar,\overline{\nu}}\times\tP_{\gamma}\big)\neq\emptyset.
$$

\subsubsection{Preliminary lemmas}

We will now make two simple (but crucial) observations that will be at the heart of our proof.

\begin{lemm}\label{r:tubesize} There exists $\kappa_0>0$ small enough (depending on $\overline{\nu}$, $\delta$, $\Lambda$ and $V_\infty$) such that, for $\hbar$ small enough, for any point $\rho\in \supp\big(\Theta_{\Lambda,\hbar,\overline{\nu}}\times\tP_{\gamma}\big)$ and any $|t|\leq \kappa_0|\log\hbar|$, one has
$$d(g^t(\rho),\Lambda^\delta)\leq \hbar^{\overline{\nu}/2}.$$
In particular, $\Lambda_n\subset W^n$.
\end{lemm}

The proof of this lemma derives from the following observation. Any point in
$\rho\in \text{supp}\big(\Theta_{\Lambda,\hbar,\overline{\nu}}\times\tP_{\gamma}\big)$
is at distance $\leq 2\hbar^{\overline{\nu}}$ from
$\Lambda^\delta$. Due to the hyperbolicity assumption, the distance from $\Lambda^\delta$ can grow at most
exponentially with time: there is a uniform $0<\lambda<1$ such that
\begin{equation}\label{e:hypdispersion}
d(g^t(\rho),\Lambda^\delta)\leq
C\hbar^{\overline{\nu}}\,\lambda^{|t|}\,,\quad \forall t\in\IR\,.
\end{equation}

This is an important property as it will allow us to apply hyperbolic dispersive estimates to cylinders in $\Lambda_n$ -- see paragraph~\ref{pa:hypdispest}. If we had chosen a larger ``tube'' around $\Lambda$, our argument would a priori not work as we will need to work with logarithmic times in $\hbar$ -- see paragraph~\ref{pa:submult}. We will also need the following feature of cylinders in $\Lambda_n$.

\begin{lemm}\label{r:bowenruelle}  There exists $\kappa_0>0$ small enough (depending on $\overline{\nu}$, $\delta$,
$\Lambda$ and $V_\infty$) such that, for $\hbar$ small enough, any $n\leq[\kappa_0|\log\hbar|]$, any $\gamma\in \Lambda_n$ and any $\rho\in\supp(\tP_{\gamma})$, one has
$$-\int_0^{nn_0-1}a\circ g^{s-nn_0}(\rho)ds\leq (nn_0-1)\beta + \mathcal{O}(1).$$ 
\end{lemm}

\begin{proof} The proof relies on lemma~\ref{l:BowenRuelle}. 
Choose $\rho\in\text{supp}(\tP_{\gamma})$. By definition of $\Lambda_n$, there exists
$\rho_{\gamma}\in
\text{supp}(\Theta_{\Lambda,\hbar,\overline{\nu}}\times\tP_{\gamma})$. The diameter
of the open cover has been selected to be smaller than
$\tilde{\epsilon}_0/2$, where $\tilde{\epsilon}_0$ is the parameter of lemma~\ref{l:BowenRuelle}. Hence, since
$g^{-k}(\rho)$ and $g^{-k}(\rho_\gamma)$ belong to the same open sets
$V_{a_k}$ for
all times $k=1,\ldots,nn_0$, we have
$$\forall 1\leq k\leq n_0n,\ d\left(g^{-k}(\rho),g^{-k}\left(\rho_{\gamma}\right)\right)\leq\frac{\tilde{\epsilon}_0}{2}.$$
Since $\rho_{\gamma}$ is at distance $\leq 2\hbar^{\overline{\nu}}$
from $\Lambda^\delta$, we can choose a point
$\tilde{\rho}_{\gamma}\in\Lambda^\delta$ such that
$d(\rho_\gamma,\tilde{\rho}_{\gamma})\leq
2\hbar^{\overline{\nu}}$. For $\kappa_0$ small enough, one gets $d(g^{-t}(\rho_\gamma),g^{-t}(\tilde{\rho}_\gamma))\leq
C\hbar^{\overline{\nu}/2}$ for all $0\leq t\leq \kappa_0|\log\hbar|$ -- see property~\eqref{e:hypdispersion}. As
a consequence, for $\hbar$ small enough,
$$\forall 1\leq k\leq nn_0,\quad
d\left(g^{-k}(\rho),g^{-k}\left(\tilde{\rho}_{\gamma}\right)\right)\leq\tilde{\epsilon}_0.
$$
Using lemma~\ref{l:BowenRuelle}, we deduce that
\begin{equation}\label{l:Bowenproof}
-\int_0^{nn_0-1}a\circ
g^{s-nn_0}(\rho)ds\leq\beta(nn_0-1)+ \mathcal{O}(1).
\end{equation}
As in the previous lemma, if we want to work with logarithmic times in
$\hbar$, we need to have a tube of size $\hbar^{\overline{\nu}}$
around $\Lambda$ in order to obtain a remainder uniform w.r.t. $\hbar$. 

\end{proof}

We underline that, in both lemmas, our choice of $\kappa_0>0$ depends on $M$, on $\Lambda$ and on our choice of open cover, of $n_0$ and of $\overline{\nu}$.

\subsection{Proof of Theorem~\ref{t:maintheo1}} 

\label{ss:proof}

We are now in the position to give the proof of our main result. Our
strategy is to prove a positive lower bound for the norm
$$\Big\|\sum_{\gamma\in \Lambda_n^c}\tPi_{\gamma}\psi_{\hbar}\Big\|,$$
where $\Lambda_n^c$ is the complementary of $\Lambda_n$ in
$\overline{W}^n$ and $n$ is a ``short logarithmic time''. It will roughly say that a positive part of the mass of $\psi_{\hbar}$ is far from $\Lambda$.

We will first use a hyperbolic
dispersive estimate~\cite{An08, NZ09, Sch10} in order to obtain a
lower bound for a similar quantity corresponding to cylinders of
length $kn$ -- see paragraph~\ref{pa:hypdispest}, with $k\gg 1$ fixed
($kn$ is a ``large logarithmic time''). Then, by a subadditive
argument (paragraph~\ref{pa:submult}), we will derive the desired
lower bound for cylinders of length $n$. Finally, we show in
paragraph~\ref{pa:semiclapprox} how to derive
Theorem~\ref{t:maintheo1} from this lower bound.

\subsubsection{Different scales of times}\label{pa:timescales}

First, we select open covers $\cV$ and $\cW^{(n_0)}$ as in
paragraph~\ref{pa:toppress}, in particular the diameter of $\cV$ is
small enough to get the requirements of remark~\ref{r:diameter}.

We will then fix some $\kappa_0>0$ small enough, so that the bound of
lemma~\ref{r:tubesize} applies, and also such that 
the quantum evolution of observables supported in the energy slab
$\cE^\delta$ is under control for times $|t|\leq \kappa_0n_0|\log\hbar|$
(see subsection~\ref{pa:egorov} on this matter). We then
introduce a ``short'' logarithmic time
\begin{equation}\label{e:Ehrenfest}
n(\hbar):=\left[\kappa_0|\log\hbar|\right]. 
\end{equation}
In particular, the arguments of
lemma~\ref{r:bowenruelle} and of paragraphs~\ref{pa:submult}
and~\ref{pa:semiclapprox} will be valid for $0\leq n\leq
n(\hbar)$. The choice of $\kappa_0$ depends on the open cover $\cV$,
on the damping function $a$, on $n_0$, on $\delta$ (the size of the energy slab we work on) and on the exponent $\overline{\nu}$ used to define $\Theta_{\hbar,\Lambda,\overline{\nu}}$.

We fix $k\geq 2$ a large positive integer,  satisfying
$k\kappa_0>\frac{d}{n_0P_0}$ -- see paragraph~\ref{pa:hypdispest}. We
will then define a second (``large'')
logarithmic time $kn(\hbar)$.

We will omit the dependence $n(\hbar)=n$ in $\hbar$ to avoid heavy notations.

\begin{rema} We underline that the different parameters we have introduced so far (namely $n_0$, $\delta$, $\kappa_0$, $k$, $P_0$ and the open cover) are chosen in a way that depends only on $\Lambda$, $a$ and $\overline{\nu}$. They will not depend on our choice of sequence $\psi_{\hbar}$.
\end{rema}

\subsubsection{Using hyperbolic dispersive estimates}\label{pa:hypdispest}

The first step of our proof is to use the
property~\eqref{e:nonsapartition} (still valid for ``large''
logarithmic times) and the fact that $\psi_\hbar$ is an eigenmode of
$\cU_{\hbar}$, in order to write
\begin{equation}\label{e:proofstep0}
\sum_{\Gamma\in\overline{W}^{kn}}\langle\Pi_{\Gamma}\psi_{\hbar},\psi_{\hbar}\rangle=\exp\left(-\frac{\imath knn_0z(\hbar)}{\hbar}\right)+\mathcal{O}(\hbar^{\infty}).
\end{equation}
Here we have implicitly used the fact that the eigenstate $\psi_\hbar$
is microlocalized on the energy layer $\cE_{1/2}=S^*M$.  
Then, we split the above sum using the decomposition of $\overline{W}^{kn}$ as
$$\overline{W}^{kn}=\Lambda_n^k\sqcup\left(\Lambda_n^k\right)^c,$$
where $\Lambda_n^k=\{\Gamma^0\Gamma^1\ldots\Gamma^{k-1}:\ \forall 0\leq j\leq k-1,\ \Gamma^j\in\Lambda_n\}$ and $\left(\Lambda_n^k\right)^c$ is the complementary of $\Lambda_n^k$ in $\overline{W}^{kn}$. We find then
\begin{equation}\label{e:proofstep1}
\sum_{\Gamma\in\Lambda_n^k}\langle\Pi_{\Gamma}\psi_{\hbar},\psi_{\hbar}\rangle+\sum_{\Gamma\in(\Lambda_n^k)^c}\langle\Pi_{\Gamma}\psi_{\hbar},\psi_{\hbar}\rangle=\exp\left(-\frac{\imath knn_0z(\hbar)}{\hbar}\right)+\mathcal{O}(\hbar^{\infty}).
\end{equation}
We will now use a hyperbolic dispersion estimate to bound the sum over
$\Lambda_n^k$ which is a subset of $W^{nk}$ -- see
lemma~\ref{r:tubesize}. We are almost in the situation of~\cite[\S
7.2]{NZ09}, except that our generator $\cP(\hbar,z)$ is
nonselfadjoint. Still, like in~\cite{Sch10}, we can use the
strategy of~\cite[Sec.4]{NZ09} by taking into account the
nonselfadjoint contribution in the WKB Ansatz.
The output is that, for
every $k\geq 2$, there exist constants $C_k>0$ and $\hbar_k>0$
(depending on $k$, on $a$, on the choice of the partition and on
$\Lambda$) such that, for any $\hbar\leq\hbar_k$ and any cylinder
$\Gamma=\alpha_0\cdots\alpha_{nk-1}\in W^{nk}$, the following
hyperbolic dispersive estimate holds:
\begin{equation}\label{e:HDE0}
\big\|\Pi_{\alpha_0\cdots\alpha_{nk-1}}\big\|_{L^2\to L^2}\leq
  C_k\,\hbar^{-\frac{d}{2}}\,(1+\mathcal{O}(\epsilon))^{knn_0}\,\prod_{j=0}^{nk-1}\,\sup_{\rho\in
  V_{\alpha_j}\cap \Lambda^\delta}\exp\Big(\int_0^{n_0}\big( 1/2 \log J^u -a)\circ g^t(\rho) \,dt\Big)\,,
\end{equation}
where the constant involved in $\mathcal{O}(\epsilon)$ depends only on
the manifold and on $a$. Recall $\epsilon$ is an upper bound on the diameter of the partition $\mathcal{V}$. Summing over all cylinders $\Gamma\in W^{nk}$ and using the
assumption \eqref{e:assumpress}, we obtain, for $\hbar$ small enough,
\begin{equation}\label{e:HDE1}
\sum_{\Gamma\in W^{nk}}\big\|\Pi_{\Gamma}\psi_{\hbar}\big\|\leq
  C_k(1+\mathcal{O}(\epsilon))^{knn_0}\,e^{knn_0(\beta-P_0)}\,\hbar^{-\frac{d}{2}}+\mathcal{O}(\hbar^{\infty})\,,
\end{equation}
which is the adaptation of the last
upper upper bound in~\cite[Sec.7]{NZ09} to our nonselfadjoint setting.
Lemma~\ref{r:tubesize} shows that $\Lambda_n\subset W^n$, so
the above sum can be restricted to $\Lambda_n^k$:
\begin{equation}\label{e:HDE2}
\sum_{\Gamma\in \Lambda_{n}^k}\big\|\Pi_{\Gamma}\psi_{\hbar}\big\|\leq
  C_k(1+\mathcal{O}(\epsilon))^{knn_0}\,e^{knn_0(\beta-P_0)}\,\hbar^{-\frac{d}{2}}+\mathcal{O}(\hbar^{\infty})\,,
\end{equation}

\begin{rema}\label{r:NZS} Let us say a few words on the proof of the crucial hyperbolic dispersive estimate~\eqref{e:HDE0}. First, we observe that any normalized 
state $v_{\hbar}$ microlocalized near the energy layer can be locally decomposed into Lagrangian states. Precisely, \emph{in a local chart} 
$f_l: V_l\subset M \rightarrow B(0,\epsilon)\subset\mathbb{R}^d$, one can represent (modulo $\cO_{L^2}(\hbar^{\infty})$) $v_{\hbar}$ as an integral of the form
$$
(2\pi\hbar)^{-\frac{d}{2}}\int_{B(0,2)} \tilde{v}_{\hbar,l}(\eta)e^{\frac{\imath \langle y,\eta\rangle}{\hbar}}d\eta,$$
where, for each ``momentum'' $\eta\in B(0,2)$, the function
$\tilde{v}_{\hbar,l}(\eta)$ is smooth and compactly supported in the variable $y\in
B(0,\eps)$ --- see for instance~\cite[\S 7]{NZ09}. Translating back to the manifold, it gives us a representation of 
$v_{\hbar}$ as a superposition of Lagrangian states. The prefactor $\hbar^{-\frac{d}{2}}$ in this decomposition is responsible for the appearance of 
$\hbar^{-\frac{d}{2}}$ in the upper bound~\eqref{e:HDE0}. Thus, in order to prove our estimate, it ``remains'' to find uniform upper bounds for the norms of 
$\Pi_{\alpha_0\cdots\alpha_{nk-1}}(a_{\hbar}e^{\frac{\imath S}{\hbar}})$, where $\left(a_{\hbar}e^{\frac{\imath S}{\hbar}}\right)_{\hbar\rightarrow 0}$ is a sequence 
of Lagrangian states microlocalized near $S^*M$ (given by the Fourier decomposition described above).

This uniform upper bounds can be obtained thanks to a careful WKB procedure. 
The difficulty comes from the fact that we have to deal with quantum evolution up to order $knn_0\asymp \mathcal{K}|\log\hbar|$ with $\mathcal{K}>0$ arbitrarly large. 
In particular, it could be delicate to represent the 
evolved state in a simple formula, because the involved Lagrangian leaves will spread over the manifold under the evolution. Here, the operator $\Pi_{\alpha_0\cdots\alpha_{nk-1}}$ does not only evolve 
the state up to large logarithmic times but it also cuts the phase space into small pieces, thanks to the cutoff operators $\pi_{\alpha_j}$ that we have inserted every time $n_0$ of the evolution. 
Due to this localization, it turns out that one can obtain a ``simpler'' description (through the WKB procedure) of the Lagrangian state evolved by $\Pi_{\alpha_0\cdots\alpha_{nk-1}}$. 
This can be done up to large logarithmic times provided we choose a good family of Lagrangian states. This property was first observed in~\cite{An08} and then used in several other situations~\cite{AN07, NZ09, Sch10}.

There is a natural choice of Lagrangian states which is associated to
the vertical bundle of the energy layer. These particular states were
used by Anantharaman and Nonnenmacher in a selfadjoint
setting~\cite{An08, AN07} and also by Schenck in~\cite{Sch10} in the
context of the damped wave equation. In these references, these
Lagrangian states remain under control up to large logarithmic times,
due to the global structure of the geodesic flow (it was supposed to
be Anosov). Indeed, the Anosov hypothesis implied that the associated Lagrangian submanifolds become uniformly close to the unstable foliation and that they do not develop caustics under the evolution (thanks to the absence of conjugate points) --- see~\cite[\S 4]{Sch10} for details.

Even if we consider the same equation, our situation differs from the one considered by Schenck in~\cite{Sch10},
because we do not make any global assumption on the geodesic flow: we
only assume it to be hyperbolic on $\Lambda$. Hence, we cannot
{\em a priori}  use the
same decomposition, because our dynamical assumptions do
not forbid the existence of conjugate points or caustics. Instead, we
may consider the more flexible Fourier decomposition introduced by
Nonnenmacher and Zworski in~\cite{NZ09}. The Lagrangian leaves
involved in this decomposition are transversal to the stable
manifolds, and therefore remain under control up to large logarithmic times
--- see~\cite[\S 5.1 and 7.1]{NZ09} for details. 

Thus, we use the Fourier decomposition of~\cite{NZ09} and we follow carefully their proof in order to prove the hyperbolic estimate~\eqref{e:HDE0}. The main difference with this reference is that we have to take into account the damping
function in the WKB procedure (like in~\cite{Sch10}). This implies that the term in the upper bound is in our setting of the form 
$$\prod_{j=0}^{nk-1}\,\sup_{\rho\in
  V_{\alpha_j}\cap \Lambda^\delta}\exp\Big(\int_0^{n_0}\big( 1/2 \log J^u -a)\circ g^t(\rho) \,dt\Big)$$ 
and not $\prod_{j=0}^{nk-1}\,\sup_{\rho\in
  V_{\alpha_j}\cap \Lambda^\delta}\exp\Big(\int_0^{n_0}\big( 1/2 \log J^u)\circ g^t(\rho) \,dt\Big)$ as in~\cite[\S 7]{NZ09}.

Following this strategy, we obtain
hyperbolic dispersion estimates for cylinders that always remain in a small vicinity of
the invariant hyperbolic set\footnote{In~\cite{NZ09}, the hyperbolic estimates were valid
  for cylinders in a small vicinity of the trapped set -- see section
  $7$ of this reference.} $\Lambda$, meaning the cylinders in $W^{kn}$.
\end{rema}

\begin{rema}
The constant $C_k$ and $\hbar_k$ involved in the hyperbolic dispersive estimate above can be chosen independently of the sequence $\psi_{\hbar}$.
\end{rema}

As was mentionned in remark~\ref{r:diameter}, the diameter $\eps$ of our initial cover was chosen small
enough to have the factor $(1+\mathcal{O}(\epsilon))\leq e^{\frac{P_0}{2}}$. 

As mentioned in \S\ref{pa:timescales}, we choose
$k\kappa_0>\frac{d}{n_0P_0}$, so that the factor
$\hbar^{-\frac{d}{2}} e^{-knn_0\frac{P_0}{2}}=o(1)$. Using the
assumption~\eqref{e:spectralpara} on $z(\hbar)$ and the fact that 
the time $knn_0=\cO(|\log\hbar|)$, we derive
$$
\Big\|\sum_{\Gamma\in\Lambda_n^k}\Pi_{\Gamma}\psi_{\hbar}\Big\|\leq
\sum_{\Gamma\in\Lambda_n^k}\|\Pi_{\Gamma}\psi_{\hbar}\|=
o\left(e^{knn_0\frac{\text{Im}\ z(\hbar)}{\hbar}}\right)\quad
\text{when }\hbar\to 0\,.
$$
Comparing this with the estimate \eqref{e:proofstep1}, we get
the following lower bound when $\hbar\to 0$:
\begin{equation}\label{e:proofstep2}
\Big\|\sum_{\Gamma\in(\Lambda_n^k)^c}\Pi_{\Gamma}\psi_{\hbar}\Big\|\geq
e^{knn_0\frac{\text{Im}\ z(\hbar)}{\hbar}} \big(1+o(1)\big)\,.
\end{equation}
This lower bound concerns the large logarithmic time $knn_0$, for
which the operators $\Pi_\Gamma$ or $\tPi_\Gamma$ cannot be analyzed
in terms of pseudodifferential calculus.

\subsubsection{Subadditivity property}\label{pa:submult} 
We will now show that the left hand side of~\eqref{e:proofstep2}
satisfies a kind of ``subadditive'' property\footnote{A similar
  property already appeared in the selfadjoint case treated
  in~\cite[\S2.2]{An08}.} for logarithmic times --- see Eq.~\eqref{e:proofstep4}. 
For that purpose, we decompose $(\Lambda_n^k)^c$ into
$$
(\Lambda_n^k)^c = \bigsqcup_{j=0}^{k-1} \left\{\Gamma=\Gamma^0\ldots\Gamma^j\ldots\Gamma^{k-1}:\forall i<j,\ \Gamma^i\in\overline{W}^n;\ \Gamma^j\in\Lambda_n^c;\ \forall i>j,\ \Gamma^i\in\Lambda_n\right\},
$$
and accordingly
\begin{equation}\label{e:sumsplit}
\sum_{\Gamma\in(\Lambda_n^k)^c}\Pi_{\Gamma}=
\sum_{j=0}^{k-1} 
\big(\sum_{\Gamma^{j+1},\ldots,\Gamma^{k-1}\in
  \Lambda_n} \Pi_{\Gamma^{k-1}}\cdots\cdots\Pi_{\Gamma^{j+1}}\big)
\big(\sum_{\Gamma^j\in \Lambda_n^c} \Pi_{\Gamma^{j}}\big)  
\big(\sum_{\Gamma^0,\ldots,\Gamma^{j-1}\in
  \overline{W}^n} \Pi_{\Gamma^{j-1}}\cdots\cdots\Pi_{\Gamma^{0}}\big)\,.
\end{equation}
Using this equality and property~\eqref{e:nonsapartition}, we are lead to
\begin{equation}\label{e:crudeupperbound}
\Big\|\sum_{\Gamma\in (\Lambda_n^k)^c}\Pi_{\Gamma}\,\psi_{\hbar}\Big\|\leq
\sum_{j=0}^{k-1} \Big\|\sum_{\gamma\in\Lambda_n}\Pi_{\gamma}\Big\|^{k-j-1}\,\Big\|\sum_{\gamma\in\Lambda_n^c}\Pi_{\gamma}\psi_{\hbar}\Big\|\,e^{jnn_0\frac{\text{Im}(z(\hbar))}{\hbar}}+\mathcal{O}(\hbar^{\infty})\,.
\end{equation}
We will show in section~\ref{s:largesums} (more precisely in
Eq.~\eqref{e:calderon}) that there exists a constant $c>0$, such
that for $\hbar$ small enough one has
$$
\Big\|\sum_{\gamma\in\Lambda_n}\Pi_{\gamma}\Big\| \leq c\, e^{nn_0\beta}\,.
$$
This bound uses the fact that we uniformly control the averaged
damping on cylinders of $\Lambda_n$, see lemma~\ref{r:bowenruelle};
in particular it uses the assumption~\eqref{e:ptylambda}. 

\begin{rema} In our argument below, we will crucially use the fact that the previous bound is $ce^{nn_0\beta}$ and not $ce^{n(n_0\beta+\tilde{\epsilon})}$ (even an arbitrary small $\tilde{\epsilon}>0$ is not a priori be sufficient for our proof). For that purpose, it was important to restrict ourselves to cylinders of trajectories that remain very close to the set $\Lambda$. If we have used all cylinders in $W^n$ (instead of $\Lambda_n$), we would have get a bound of order $ce^{n(n_0\beta+\tilde{\epsilon})}$ which would have not been sufficient for the end of our proof.

\end{rema}

Then, the assumption~\eqref{e:spectralpara} on $z(\hbar)$ shows that the
above right hand side is smaller than $c\, e^{nn_0\frac{\text{Im}\ z}{\hbar}}(1+o(1))$, 
therefore~\eqref{e:crudeupperbound} becomes
\begin{equation}\label{e:proofstep4}
\Big\|\sum_{\Gamma\in(\Lambda_n^k)^c}\Pi_{\Gamma}\psi_{\hbar}\Big\|\leq
c^k\,k(1+o(1))\, e^{(k-1)nn_0\frac{\text{Im}\ z}{\hbar}} \Big\|\sum_{\gamma\in\Lambda_n^c}\Pi_{\gamma}\psi_{\hbar}\Big\|+\mathcal{O}(\hbar^{\infty})\,.
\end{equation}
Combining this inequality with the lower bound~\eqref{e:proofstep2}, one obtains
\begin{equation}\label{e:proofstep5}
\Big\|\sum_{\gamma\in\Lambda_n^c}\Pi_{\gamma}\psi_{\hbar}\Big\| \geq (c^kk)^{-1}e^{nn_0\frac{\text{Im}\ z}{\hbar}}(1+o(1)) +\mathcal{O}(\hbar^{\infty}).
\end{equation}
This lower bound is our desired lower bound for a
``short'' logarithmic time. 
\begin{rema} \label{r:paradepend}
We underline again that the constants $c$ and $k$ do not depend on the sequence $(\psi_{\hbar})$, but only on $\delta$, $P_0$, $n_0$, the choice of open cover and $\kappa_0$. Thus, it depends only on $\Lambda$, $a$ and $\overline{\nu}$ and it will be this constant $(c^kk)^{-1}$ that will play the role of $c_{\Lambda, a,\overline{\nu}}$ in~\eqref{e:contradictionlim}.
\end{rema}

\subsubsection{Using semiclassical calculus}
\label{pa:semiclapprox}

Since $\psi_\hbar$ is an eigenstate of $\cU_\hbar$, the inequality~\eqref{e:proofstep5} can be rewritten as
$$
\big\|\sum_{\gamma\in\Lambda_n^c}\tPi_{\gamma}\psi_{\hbar}\big\|_{L^2(M)}
\geq (c^kk)^{-1}(1+o(1)) + \mathcal{O}(\hbar^{\infty}).$$
Using the observations of paragraph~\ref{pa:largesums}, and the
fact that $\kappa_0$ has been chosen small enough, for
$n=[\kappa_0|\log\hbar|]$ the operator
$$\tPi_{\Lambda_n^c}:=\sum_{\gamma\in\Lambda_n^c}\tPi_{\gamma}$$ is
approximately the quantization of the symbol
$\tP_{\Lambda_n^c}:=\sum_{\gamma\in\Lambda_n^c}\tP_\gamma$, which
belongs to the
symbol class $S^{-\infty,0}_{\overline{\nu}'}(T^*M)$ for some
$\overline{\nu}'\in (0,1/2)$. Using also the composition rule in $\Psi^{-\infty,0}_{\overline{\nu}'}(M)$, we get the bound
$$
\big\langle\Op_{\hbar}(\tP_{\Lambda_n^c}^2)\psi_{\hbar},\psi_{\hbar}\big\rangle
\geq (c^kk)^{-2}(1+o(1)) + \mathcal{O}(\hbar^{\nu_0}),$$
for some $\nu_0>0$. 
By construction, the function $\tP_{\Lambda_n^c}$ takes values in $[0,1]$. Because the quantization $\Op_{\hbar}$ is approximately positive for symbols in this
class --- see paragraph~\ref{p:antiwick} --- one finds that
$$ 
\big\langle\Op_{\hbar}(\tP_{\Lambda_n^c})\psi_{\hbar},\psi_{\hbar}\big\rangle
\geq (c^kk)^{-2}(1+o(1)) + \mathcal{O}(\hbar^{\nu_0}).
$$

\begin{rema} The value of $\nu_0>0$ can be different from the one appearing above: we have just kept the largest remainder term. 
\end{rema}

We now split the above left hand side into two parts, using the cutoff function $\Theta_{\hbar,\Lambda,\overline{\nu}}$. It remains to estimate
$$(A):=\big\langle\Op_{\hbar}\big(\tP_{\Lambda_n^c}\,(1-\Theta_{\hbar,\Lambda,\overline{\nu}})\big)\psi_{\hbar},\psi_{\hbar}\big\rangle,$$
and
$$(B):=\big\langle\Op_{\hbar}\big(\tP_{\Lambda_n^c}\,\Theta_{\hbar,\Lambda,\overline{\nu}}\big)\psi_{\hbar},\psi_{\hbar}\big\rangle.$$
Using again the fact that $\Op_{\hbar}$ is almost positive, and
that $\tP_{\Lambda_n^c}\leq 1$, one obtains the bound
$$(A)\leq
\big\langle\Op_{\hbar}(1-\Theta_{\hbar,\Lambda,\overline{\nu}})\psi_{\hbar},\psi_{\hbar}\big\rangle+\mathcal{O}(\hbar^{\nu_0}).$$
On the other hand, the definition of $\Lambda_n^c$ implies that $(B)=0$. This leads to 
$$\liminf_{\hbar\rightarrow0}\big\langle\Op_{\hbar}(1-\Theta_{\Lambda,\hbar,\overline{\nu}})
\psi_{\hbar},\psi_{\hbar}\big\rangle \geq (c^kk)^{-2},
$$
which concludes the proof of Theorem~\ref{t:maintheo1}. The lower bound depends only on $\Lambda$, $a$ and $\overline{\nu}$ --- see remark~\ref{r:paradepend}.

\section{Long products of pseudodifferential operators}\label{s:largesums}

In this section, we describe some properties of long products of pseudodifferential operators evolved under the quantum propagator. For that purpose, we recall first a few facts on the Egorov property for nonselfadjoint operators and then we apply them to our problem.

\subsection{Egorov property for long times}\label{pa:egorov}

In this paragraph, we recall an Egorov property for times of order
$\kappa_0|\log\hbar|$, where $\kappa_0$ is a small enough constant
that we will not try to optimize. Consider $q_1$ and $q_2$ two symbols
belonging to $S^{0,0}(T^*M)$ (for the sake of simplicity, we also
assume that these symbols depend smoothly on $\hbar\in (0,1]$). In
this article, we will use the symbols $q_i$ equal to
$\sqrt{2z(\hbar)}a$, $-\sqrt{2\bar{z}(\hbar)}a$ or $0$ --- see paragraph~\ref{pa:largesums} below.

\subsubsection{The case of fixed times}

We consider a smooth function $b$ on $T^*M$ which is compactly supported in a neighborhood of $S^*M$, say 
$\text{supp}(b)\subset\{(x,\xi):\|\xi\|^2\in[1/2,3/2]\}$ and which belongs to $S^{-\infty,0}(T^*M)$. The following operator is a pseudodifferential operator, for every $t\in\mathbb{R}$,
$$B(t,b)=\left(e^{-\frac{\imath t}{\hbar}\left(-\frac{\hbar^2\Delta}{2}-\imath\hbar\Op_{\hbar}(q_1)\right)}\right)^*\Op_{\hbar}(b)
e^{-\frac{\imath t}{\hbar}\left(-\frac{\hbar^2\Delta}{2}-\imath\hbar\Op_{\hbar}(q_2)\right)}.$$
We briefly recall how such a fact can be proved by a direct adaptation of the arguments used in the selfadjoint case~\cite{DS, EZ, BoRo02, Roy10}. Take $q=\overline{q_1}+q_2$, and 
introduce, for $t,s\in\mathbb{R}$, the symbol
$$B_t(s):=b\circ g^{t-s}\exp\left(-\int_0^{t-s}q\circ g^{\tau}d\tau\right).$$
To alleviate our notations, we call
$$\U_{\hbar}^s(q_i):=e^{-\frac{\imath
    s}{\hbar}\left(-\frac{\hbar^2\Delta}{2}-\imath\hbar\Op_{\hbar}(q_i)\right)},\quad
i=1,2\,,
$$
so that the operator $B(t,b) = \left(\U_{\hbar}^t(q_1)\right)^*\Op_{\hbar}(b)\,\U_{\hbar}^t(q_2)$.
Fixing $t$, we then introduce the auxiliary operators
$$R(\hbar,s)=\left(\U_{\hbar}^s(q_1)\right)^*\Op_{\hbar}(B_t(s))\,\U_{\hbar}^s(q_2).$$
Like in the classical proof of the Egorov Theorem (i.e. in the selfadjoint case), one can compute the derivative of $R(\hbar,s)$:
$$\frac{d}{ds}\left(R(\hbar,s)\right)=
(\mathcal{U}^s_{\hbar}(q_1))^*\left(\frac{\imath}{\hbar}\left[-\frac{\hbar^2\Delta}{2},\Op_{\hbar}(B_t(s))\right]-\Op_{\hbar}(q_1)^*\Op_{\hbar}(B_t(s))-
\Op_{\hbar}(B_t(s))\Op_{\hbar}(q_2)\right)\mathcal{U}^s_{\hbar}(q_2)$$
$$\hspace{4cm}-(\mathcal{U}^s_{\hbar}(q_1))^*\big(\Op_{\hbar}\left(\{p_0,B_t(s)\}\right)-\Op_{\hbar}(B_t(s)(\overline{q_1}+q_2))\big)\mathcal{U}^s_{\hbar}(q_2).$$
We integrate this equality between $0$ and $t$~\cite{BoRo02}:
$$\left(\U_{\hbar}^t(q_1)\right)^*\Op_{\hbar}(b)\,\U_{\hbar}^t(q_2)=\Op_{\hbar}\left(b\circ g^{t}e^{-\int_0^{t}q\circ g^{\tau}d\tau}\right)+\int_0^t(\mathcal{U}^s_{\hbar}(q_1))^*\tilde{R}(\hbar,s)\mathcal{U}^s_{\hbar}(q_2)ds,$$
where $\tilde{R}(\hbar,s)$ is a pseudodifferential operator in $\Psi^{-\infty,-1}(M)$ thanks to pseudodifferential rules. Proceeding by induction and using pseudodifferential calculus perfomed locally on each chart~\cite{DS, EZ} (respectively Chapter $7$ and $4$) and the fact 
that $\mathcal{U}^s_{\hbar}(q_2)$ is a bounded operator (with a norm depending\footnote{It is in fact bounded by a constant of order $e^{|s|\|q_2\|_{\infty}}$.} on $q_2$ and $s$), one in fact finds that 
$\left(\U_{\hbar}^t(q_1)\right)^*\Op_{\hbar}(b)\U_{\hbar}^t(q_2)$ is a pseudodifferential operator in $\Psi^{-\infty,0}(M)$,
\begin{equation}\label{e:Egorov1}
\left(\U_{\hbar}^t(q_1)\right)^*\Op_{\hbar}(b)\U_{\hbar}^t(q_2)=\Op_{\hbar}(\tilde{b}(t))+\mathcal{O}(\hbar^{\infty}),
\end{equation}
where $\tilde{b}(t)\sim\sum_{j\geq 0}\hbar^{j} b_j(t)$, 
$$b_0(t)=B_t(0)=b\circ g^{t}\exp\left(-\int_0^{t}(\overline{q_1}+q_2)\circ g^{\tau}d\tau\right),$$ 
and all the higher order terms $(b_j(t))_{j\geq1}$ in the asymptotic 
expansion depend on $b$, $t$, $q_1$, $q_2$ and the choice of
coordinates on the manifold. Moreover, for a fixed $t\in\mathbb{R}$,
one can verify that every term $b_j(t)$ is supported in
$g^{-t}\text{supp}(b)$. Each $b_j(t)$ can be written as
$c_j(t)\exp\left(-\int_0^{t}(\overline{q_1}+q_2)\circ
  g^{\tau}d\tau\right)$, where $c_j(t)\in S^{-\infty,0}(T^*M)$. The
Calder\'on-Vaillancourt Theorem~\cite[Chap.5]{EZ} tells us that there exist constants $C_{b,t}$ and $C_{b,t}'$ (depending on $b$, $q_1$, $q_2$, $t$ and $M$) such that
$$
 \left\|\left(\U_{\hbar}^t(q_1)\right)^*\Op_{\hbar}(b)\,\U_{\hbar}^t(q_2)\right\|_{L^2(M)\rightarrow L^2(M)}\leq C_{b,t}\|b_0(t)\|_{\infty},$$
and also
\begin{equation}\label{e:generalegorov}
 \left\|\left(\U_{\hbar}^t(q_1)\right)^*\Op_{\hbar}(b)\U_{\hbar}^t(q_2)-\Op_{\hbar}(b_0(t))\right\|_{L^2(M)\rightarrow
   L^2(M)}\leq C_{b,t}'\hbar.
\end{equation}

\subsubsection{The case of logarithmic times}

All the above discussion was done for a fixed $t\in\mathbb{R}$. In this article, we needed to apply Egorov property for long range of times of order $\kappa_0|\log\hbar|$~\cite{BoRo02, AN07}. This can be achieved as all the arguments above can be adapted if we use more general classes of symbols, i.e. $S_{\overline{\nu}}^{-\infty,0}(T^*M)$ where $\overline{\nu}<1/2$ is a fixed constant\footnote{In order to avoid too many indices, we take the same $\overline{\nu}$ as in the definition of $\Theta_{\Lambda,\hbar,\overline{\nu}}$.}.

In particular, one can show that, for $b\in S^{-\infty,0}(T^*M)$
supported near $S^*M$ as above and $\kappa_1$ small enough (depending
on the support of  $b$, on $\overline{\nu}$, on $q_1$ and on $q_2$),
the operator $B(t,b)$ is a pseudodifferential operator in
$\Psi^{-\infty,0}_{\overline{\nu}}(M)$ for all
$|t|\leq\kappa_1|\log\hbar|$. Precisely, its symbol has an asymptotic expansion of the same form as in the case of fixed times, except that for every $j\geq 0$ the symbol $c_j(t)$ belongs to $S^{-\infty,k_j}_{\overline{\nu}}(T^*M)$ for every $|t|\leq\kappa_1|\log\hbar|$, where $j-k_j$ is an increasing sequence of real numbers converging to infinity as $j\rightarrow+\infty$. 

We also mention that all the seminorms of the symbols $c_j(t)$ can be bounded uniformly for $|t|\leq\kappa_1|\log\hbar|$.  Finally, using pseudodifferential calculus (performed locally on every chart), one can verify that the following uniform estimates hold:

\begin{prop}\label{p:egorov} There exist constants $\kappa_1>0$ and $\nu_0>0$ (depending only on $q_1$, $q_2$, $\overline{\nu}$ and $M$) such that for every smooth function $b$ compactly supported in $\{(x,\xi):\|\xi\|^2\in[1/2,3/2]\}$, there exists a constant $C_b>0$ such that 
for every $ |t|\leq \kappa_1|\log\hbar|$, one has
$$ \Big\|\left(\U_{\hbar}^t(q_1)\right)^* \Oph(b)\,\U_{\hbar}^t(q_2)\Big\|_{L^2(M)\rightarrow L^2(M)}\leq C_{b}\|b_0(t)\|_{\infty},$$
and
$$ \left\|\left(\U_{\hbar}^t(q_1)\right)^* \Oph(b)\,\U_{\hbar}^t(q_2)-\Op_{\hbar}(b_0(t))\right\|_{L^2(M)\rightarrow L^2(M)}\leq C_{b}\hbar^{\nu_0}.$$
\end{prop}

\begin{rema}\label{r:U*bU}
We will mostly use evolutions involving the propagator $\cU_\hbar^t$ of
\eqref{e:propag}. Then, the expression 
$(\U_\hbar^t)^*\Op_\hbar(b) \,\U_\hbar^t$ has the form of
\eqref{e:Egorov1}, with $q_1 = q_2 = \sqrt{2z(\hbar)}a$. As a result,
in this case the principal symbol is $b_0(t)=b\circ
g^t\,e^{-2\int_0^ta\circ g^\tau\,d\tau}$.

Another operator will be used: $(\U_\hbar^t)^{-1}\Op_\hbar(b)
\,\U_\hbar^t$ also has the form \eqref{e:Egorov1}, now with
$q_1=-\sqrt{2\bar z}a$, $q_2=\sqrt{2z(\hbar)}a$. In this case, the
principal symbol $b_0(t)=b\circ g^t$.
\end{rema}
\subsection{Sums of long products of pseudodifferential operators}\label{pa:largesums}

In this paragraph, we make a few observations on ``long'' product of 
 pseudodifferential operators (with $\asymp |\log\hbar|$ factors), that we used at different stages of our proof -- e.g. in paragraphs~\ref{pa:submult} and~\ref{pa:semiclapprox}.

The open cover and the time $n_0$ of paragraph~\ref{pa:toppress} (and their corresponding quantum partition near $\mathcal{E}^{\delta}$) are fixed in this paragraph.

We would like to use the above results to show that, for $\kappa_0>0$
small enough, for $0\leq p\leq\kappa_0|\log\hbar|$ and for any subset
$X_p\subset\overline{W}^p$ of $p$-cylinders, the operator
$$\tPi_{X_p}:=\sum_{\gamma\in X_p}\tPi_{\gamma}$$
is a pseudodifferential operator, with a principal symbol in a
``good'' symbol class. Using the composition rule for pseudodifferential operators in $\Psi^{-\infty,0}_{\overline{\nu}}(M)$ and proposition~\ref{p:egorov}, there exist $\nu_0>0$ and $\kappa_0>0$ such that, for every $0\leq p\leq\kappa_0|\log\hbar|$ and for every $\gamma\in\overline{W}^p$,
$$\left\|\tPi_{\gamma} - \Op_{\hbar}(\tP_{\gamma})\right\|_{L^2(M)}=\mathcal{O}(\hbar^{\nu_0}),$$
where the remainder can be bounded uniformly for every $0\leq p\leq\kappa_0|\log\hbar|$ and for every cylinder $\gamma\in\overline{W}^p$. 

\begin{rema}
The constants $\nu_0$ and $\kappa_0$ appearing here are a priori smaller than the one from proposition~\ref{p:egorov}.
\end{rema}

This observation leads us to the bound
\begin{equation}\label{e:tPi-tP}
\big\|\tPi_{X_p}-\Oph(\tP_{X_p})\big\|_{L^2(M)}=\mathcal{O}(K^p\hbar^{\nu_0}),\qquad \tP_{X_p}:=\sum_{\gamma\in
  X_p}\tP_{\gamma}
\end{equation}
where $K=|\overline{W}|$. 
Hence, for $\kappa_0$ small enough, the remainder is of the form
$\mathcal{O}(\hbar^{\nu_0'})$ for some positive $\nu_0'>0$. We
underline that the constant in the remainder is uniform w.r.to $0\leq
p\leq\kappa_0|\log\hbar|$ and
$X_p\subset \overline{W}^p$.

We can also verify that there exists $\kappa_0>0$ small enough and
$\overline{\nu}<\overline{\nu'}<1/2$ such that the function
$\tP_{X_p}$ belongs to the symbol class
$S^{-\infty,0}_{\overline{\nu}'}(T^*M)$, and such that the seminorms
(defining this class) can be bounded uniformly w.r.to
$0\leq p\leq\kappa_0|\log\hbar|$ and $X_p\subset\overline{W}^p$.
In particular, one can apply semiclassical calculus to this operator. For instance, the Calder\'on-Vailancourt Theorem tells us that
\begin{equation}\label{e:calderonbis}\big\|\Op_{\hbar}(\tPi_{X_p})\big\|_{L^2\to L^2}= \mathcal{O}(1),\end{equation}
where the constant in the remainder is uniform w.r.to $0\leq p\leq\kappa_0|\log\hbar|$ and $X_p\subset\overline{W}^p$.

\begin{rema}\label{r:productpositime}
When proving the subadditive property, we also needed to bound
from above the norm of
$$\mathbf{Q}_{X_p}:=e^{-\frac{\imath
    pn_0\hbar\Delta}{2}}\sum_{\gamma\in X}\Pi_{\gamma}
\,,\quad\text{for a subset }X_p\subset \Lambda_p\,.
$$
Using the notations of \S\ref{pa:egorov}, this operator can be written
$$
\mathbf{Q}_{X_p} = \U_\hbar(0)^{-pn_0}\,\tPi_{X_p}\,\U_\hbar(\sqrt{2z}a)^{pn_0}
$$
Hence, using \eqref{e:tPi-tP} and the Egorov type estimate of
Proposition~\ref{p:egorov}, one obtains, for $\kappa_0$ small enough,
$$\left\|\mathbf{Q}_{X_p}-\Op_{\hbar}\big(\tP_{X_p}\circ g^{pn_0}\,e^{-\int_0^{pn_0}a\circ g^s ds}\big)\right\|_{L^2(M)}=\mathcal{O}(\hbar^{\nu_0'}),$$
for some $\nu_0'>0$. The symbol
$$\tP_{X_p}\circ g^{pn_0}\,e^{-\int_0^{pn_0}a\circ g^s ds} = 
e^{-\int_0^{pn_0}a\circ g^s ds}\sum_{\gamma\in X_p} P_{\gamma^{p-1}}\circ g^{(p-1)n_0}\ldots P_{\gamma^1}\circ g^{n_0}  P_{\gamma^0}$$
belongs to a class $S^{-\infty,0}_{\overline{\nu}'}(T^*M)$. In
particular, since $X_p\subset \Lambda_p$, one can combine lemma~\ref{r:bowenruelle} with the
Calder\'on-Vaillancourt Theorem in order to derive that, for
$\kappa_0$ small enough and for any $0\leq p\leq\kappa_0|\log\hbar|$,
one has the norm estimate
\begin{equation}\label{e:calderon}\|\mathbf{Q}_{X_p}\|_{L^2}=\mathcal{O}(e^{pn_0\beta}),
\end{equation}
where the implied constant is uniform in $p$, $X_p\subset \Lambda_p$ and depends on $a$, on the choice of the open cover and on $n_0$.
\end{rema}

\begin{rema} 
Even if we did not mention it at every stage of the proof, the remainders due to the semiclassical approximation depend on the choice of the open cover and on $n_0$ that were introduced in paragraph~\ref{pa:toppress}.
\end{rema}

\section{Pseudodifferential calculus on a manifold}

\label{a:pdo}
In this last section, we review some basic facts on semiclassical analysis that can be found for instance in~\cite{DS, EZ}.

\subsection{General facts}

Recall that we define on $\mathbb{R}^{2d}$ the following class of symbols:
$$S^{m,k}(\mathbb{R}^{2d}):=\left\{(b_{\hbar}(x,\xi))_{\hbar\in(0,1]}\in C^{\infty}(\mathbb{R}^{2d}):|\partial^{\alpha}_x\partial^{\beta}_{\xi}b_{\hbar}|
\leq C_{\alpha,\beta}\hbar^{-k}\langle\xi\rangle^{m-|\beta|}\right\}.$$
Let $M$ be a smooth Riemannian $d$-manifold without boundary. Consider a smooth atlas $(f_l,V_l)$ of $M$, where each $f_l$ is a smooth diffeomorphism from 
$V_l\subset M$ to a bounded open set $W_l\subset\mathbb{R}^{d}$. To each $f_l$ correspond a pull back $f_l^*:C^{\infty}(W_l)\rightarrow C^{\infty}(V_l)$ and a canonical 
map $\tilde{f}_l$ from $T^*V_l$ to $T^*W_l$:
$$\tilde{f}_l:(x,\xi)\mapsto\left(f_l(x),(Df_l(x)^{-1})^T\xi\right).$$
Consider now a smooth locally finite partition of identity $(\phi_l)$ adapted to the previous atlas $(f_l,V_l)$. 
That means $\sum_l\phi_l=1$ and $\phi_l\in C^{\infty}(V_l)$. Then, any observable $b$ in $C^{\infty}(T^*M)$ can be decomposed as follows: $b=\sum_l b_l$, where 
$b_l=b\phi_l$. Each $b_l$ belongs to $C^{\infty}(T^*V_l)$ and can be pushed to a function $\tilde{b}_l=(\tilde{f}_l^{-1})^*b_l\in C^{\infty}(T^*W_l)$. 
As in~\cite{DS, EZ}, define the class of symbols of order $m$ and index $k$
\begin{equation}
\label{defpdo}S^{m,k}(T^{*}M):=\left\{(b_{\hbar}(x,\xi))_{\hbar\in(0,1]}\in C^{\infty}(T^*M):|\partial^{\alpha}_x\partial^{\beta}_{\xi}b_{\hbar}|\leq C_{\alpha,\beta}\hbar^{-k}\langle\xi\rangle^{m-|\beta|}\right\}.
\end{equation}
Then, for $b\in S^{m,k}(T^{*}M)$ and for each $l$, one can associate to the symbol $\tilde{b}_l\in S^{m,k}(\mathbb{R}^{2d})$ the standard Weyl quantization
$$\Op_{\hbar}^{w}(\tilde{b}_l)u(x):=
\frac{1}{(2\pi\hbar)^d}\int_{R^{2d}}e^{\frac{\imath}{\hbar}\langle x-y,\xi\rangle}\tilde{b}_l\left(\frac{x+y}{2},\xi;\hbar\right)u(y)dyd\xi,$$
where $u\in\mathcal{S}(\mathbb{R}^d)$, the Schwartz class. Consider now a smooth cutoff $\psi_l\in C_c^{\infty}(V_l)$ such that $\psi_l=1$ close to the support of $\phi_l$. 
A quantization of $b\in S^{m,k}(T^*M)$ is then defined in the following way (see chapter $14$ in~\cite{EZ}):
\begin{equation}
\label{pdomanifold}\Op_{\hbar}(b)(u):=\sum_l \psi_l\times\left(f_l^*\Op_{\hbar}^w(\tilde{b}_l)(f_l^{-1})^*\right)\left(\psi_l\times u\right),
\end{equation}
where $u\in C^{\infty}(M)$. This quantization procedure $\Op_{\hbar}$ sends (modulo $\mathcal{O}(\hbar^{\infty})$) $S^{m,k}(T^{*}M)$ onto the space of pseudodifferential 
operators of order $m$ and of index $k$, denoted $\Psi^{m,k}(M)$~\cite{DS, EZ}. It can be shown that the dependence in the cutoffs $\phi_l$ and $\psi_l$ only appears at order 
$1$ in $\hbar$ (Theorem $18.1.17$ in~\cite{Ho} or Theorem $9.10$ in~\cite{EZ}) and the principal symbol map $\sigma_0:\Psi^{m,k}(M)\rightarrow S^{m-1,k}/S^{m-1,k-1}(T^{*}M)$ is then 
intrinsically defined. Most of the rules (for example the composition of operators, the Egorov and Calder\'on-Vaillancourt Theorems) that hold on 
$\mathbb{R}^{2d}$ still hold in the case of $\Psi^{m,k}(M)$. Because our study concerns the behavior of quantum evolution for logarithmic times in $\hbar$, a larger class of 
symbols should be introduced as in~\cite{DS, EZ}, for $0\leq\overline{\nu}<1/2$,
\begin{equation}\label{symbol}
S^{m,k}_{\overline{\nu}}(T^{*}M):=\left\{(b_{\hbar})_{\hbar\in(0,1]}\in C^{\infty}(T^*M):
|\partial^{\alpha}_x\partial^{\beta}_{\xi}b_{\hbar}|\leq C_{\alpha,\beta}\hbar^{-k-\overline{\nu}|\alpha+\beta|}\langle\xi\rangle^{m-|\beta|}\right\}.
\end{equation}
Results of~\cite{DS, EZ} can be applied to this new class of symbols. For example, a symbol of $S^{0,0}_{\overline{\nu}}(T^*M)$ gives a bounded operator on $L^2(M)$ 
(with norm uniformly bounded with respect to $\hbar$).

\subsection{Positive quantization}\label{p:antiwick}

Even if the Weyl procedure is a natural choice to quantize an observable $b$ on $\mathbb{R}^{2d}$, it is sometimes preferrable to use a quantization procedure $\Op_{\hbar}$ that satisfies the property~: $\Op_{\hbar}(b)\geq 0$ if $b\geq0$. This can be achieved thanks to the anti-Wick procedure $\Op_{\hbar}^{AW}$, see~\cite{HeMaRo}.
 For $b$ in $S^{0,0}_{\overline{\nu}}(\mathbb{R}^{2d})$, that coincides with a function on $\mathbb{R}^d$ outside a compact subset of $T^*\mathbb{R}^d=\mathbb{R}^{2d}$, one has
\begin{equation}\label{equivalence-positive-quantization}\|\Op_{\hbar}^w(b)-\Op_{\hbar}^{AW}(b)\|_{L^2}\leq C\sum_{|\alpha|\leq D}\hbar^{\frac{|\alpha|+1}{2}}\|\partial^{\alpha}db\|,
\end{equation}
where $C$ and $D$ are some positive constants that depend only on the dimension $d$.
To get a positive procedure of quantization on a manifold, one can replace the Weyl quantization by the anti-Wick one in definition~(\ref{pdomanifold}). This new choice of quantization is well defined for every element in $S^{0,0}_{\overline{\nu}}(T^*M)$ of the form $c_0(x)+c(x,\xi)$ where $c_0$ belongs to $S^{0,0}_{\overline{\nu}}(T^*M)$ and $c$ belongs to $\mathcal{C}^{\infty}_o(T^*M)\cap S^{0,0}_{\overline{\nu}}(T^*M)$.

\appendix

\section{Inverse logarithmic ``spectral gap'' under a pressure condition\\
By St\'ephane Nonnenmacher and Gabriel Rivi\`ere}
\label{a:loggap}

In this appendix, we consider the problem~\eqref{e:specprob} in the case
where the damping function $a(x)\geq 0$ does not identically
vanish. We also make the assumption that
 the {\it set of undamped trajectories} 
$$
\cN=\left\{\rho\in
  S^*M:a\circ g^t(\rho)=0,\ t\in\mathbb{R}\right\}
$$
is not empty. In this case, it is generally not known whether there
exists a strip of fixed width below the real axis without eigenvalues
of~\eqref{e:specprob}. Lebeau showed~\cite{Leb93} the existence of an
exponentially thin strip, meaning that there exists $C>0$ such that all
eigenvalues $\tau\neq 0$ satisfy
$$
\text{Im}\ \tau\leq - \frac{1}{C}e^{-C|\tau|}\,.
$$
Lebeau also constructed a geometric situation where this upper bound is sharp. Yet, it is natural to ask whether additional assumptions on the manifold $M$ and on the set $\cN$ allow to improve this upper bound. In this
appendix, we apply the techniques developed above to prove the following criterium for an inverse logarithmic gap.

\begin{theo}\label{th:log-gap} Assume the set of undamped trajectories $\cN$ is a
  hyperbolic set, and satisfies the pressure condition
\begin{equation}\label{e:negpress}
P_{top}\left(\cN,g^t,\frac{1}{2}\log J^u\right) < 0.
\end{equation}
Then, there exists a constant $C>0$ such that for the following
resolvent estimate holds:
\begin{equation}\label{e:resolv-est}
\|(-\Delta-2ia\tau-\tau^2)^{-1}\|\leq
\frac{C(\log(\Re\tau))^2}{\Re\tau},\quad \text{uniformly for }\tau\in \Big\{\Re \tau\geq C,\
|\Im \tau|\leq \frac{C^{-1}}{\log (\Re\tau)}\Big\}\,.
\end{equation}
As a consequence, there is a $\tilde{C}>0$ such that any eigenvalue $\tau_n\neq 0$ of the
problem~\eqref{e:specprob} satisfies
\begin{equation}\label{e:gap}
\Im\ \tau_n\leq - \frac{\tilde C}{\log(1+|\tau_n|)}\,.
\end{equation}
\end{theo}
This inverse logarithmic
spectral gap was recently obtained in~\cite[Thm. 5.5]{CSVW12} using a
different approach, and under the slightly stronger assumption that
$\pi(\cN)\cap\text{supp}(a)=\emptyset$ where $\pi:S^*M\rightarrow M$
is the canonical projection on $M$ (in our setting, $\cN$ is allowed to
intersect $\supp a\cap a^{-1}(0)$). However, the resolvent
estimate obtaind in ~\cite[Thm. 5.5]{CSVW12} is of order
$\frac{\log(\Re\tau)}{\Re\tau}$, which is sharper (by a
logarithmic factor) than the one we obtain above. We believe that this
loss of a logarithmic factor is due to our method of proof, and that
the upper bound $\frac{\log(\Re\tau)}{\Re\tau}$ should hold under our
conditions as well.

A similar result had been proved by Christianson in~\cite{Chr07}, under the assumption
that $\cN$ consists in a single hyperbolic closed geodesic, and extended
in~\cite{Chr11} to the case of a (single) semihyperbolic
closed geodesic\footnote{A semihyperbolic closed geodesic admits at
  least one positive
  Lyapunov exponent.} satisfying a nonresonance assumption.
In~\cite{Riv11a} the same spectral gap
was proved under the assumption that the geodesic flow on $M$ is
Anosov~\cite{KaHa}. The above Theorem thus
generalizes the results of~\cite{Chr07,Riv11a}, and it cannot be
improved without additional assumptions --- see the example
announced in~\cite{BuChr09}. 

In order to get a larger gap, one can try to make {\em global}
assumptions on the geodesic flow on $M$, for instance assume it is of
Anosov type.
It was conjectured in~\cite{Non11} that if the geodesic flow  is Anosov and $\cN$
satisfies the condition~\eqref{e:negpress}, then there should be a \emph{finite} spectral
gap, namely all eigenvalues $\tau\neq 0$ of the
problem~\eqref{e:specprob} should satisfy $\Im\ \tau\leq -\gamma$ for
some $\gamma>0$. 
We refer the reader to~\cite{Sch11,Non11} for partials results in favor of this conjecture. 

The references~\cite{Leb93,Chr09} show how to connect resolvent estimates with the
decay of the {\em energy}
$$
 E(v(t)) \defeq \frac12 \big(\|\nabla v(t)\|^2 + \|\partial_t v(t)\|^2\big)
$$
of a wave $v(x,t)$ satisfying \eqref{e:DWE}. With our dynamical conditions one
obtains a stretched exponential decay (see \cite[Cor. 5.2]{CSVW12}):
\begin{coro}\label{cor:energy-decay}
Assume the same geometric conditions as in Thm~\ref{th:log-gap}. For
any $s>0$ there exists $C_s>0$, such that for
any initial data $(v(0),\partial_t v(0))\in H^{s+1}(M)\times H^s(M)$,
the energy of the wave $v(t)$ solving \eqref{e:DWE} with those data
satisfies
$$
\forall t\geq 0,\qquad E(v(t))\leq C_s\,e^{-t^{1/2}/C_s}\,\big(\|v(0)\|^2_{H^{s+1}} + \|\partial_tv(0)\|^2_{H^{s}} \big)\,.
$$
\end{coro}
\begin{rema}
The undamped set $\cN$ can be ``lifted'' to nearby energy
shells, and we will often consider $\cN^\delta$ defined as in \eqref{e:Lambda^delta}. Due to the homogeneity of the
geodesic flow, the condition \eqref{e:negpress}
is satisfied on all nonzero energy shells when it is on $S^*M=p_0^{-1}(1/2)$.
\end{rema}

We now give the proof of Theorem~\ref{th:log-gap}.

\begin{proof}

Using the semiclassical notations of the introduction, we need to
establish the existence of constants $\delta_0>0$, and $C>0$ such
that, for $\hbar>0$ small enough,
\begin{equation}\label{e:resolvent}
\forall z\in \left[\frac{1}{2}-\delta_0,\frac{1}{2}+\delta_0\right]+\imath\left[-C^{-1}\frac{\hbar}{|\log\hbar|},
C^{-1}\frac{\hbar}{|\log\hbar|}\right],\ \left\|(\mathcal{P}(\hbar,z)-z)^{-1}\right\|_{\mathcal{L}(L^2(M))}\leq\frac{C|\log\hbar|^2}{\hbar}.
\end{equation}
Translating back to the original setting of \eqref{e:specprob}, this
resolvent estimate implies~\eqref{e:resolv-est}.

In order to prove~\eqref{e:resolvent}, we proceed
by contradiction. Namely, we assume that there exist a sequence of
parameters $(\hbar_l\searrow 0)_{l\in\IN}$, of spectral parameters $z(\hbar_l)\in\mathbb{C}$ and
of normalized {\em quasimodes} $\psi_{\hbar_l}\in L^2(M)$, so that, when $l$ tends to infinity,
\begin{equation}\label{e:contradiction}
\begin{split}
\mathcal{P}(z(\hbar_l),\hbar_l)\psi_{\hbar_l} & =z(\hbar_l)\psi_{\hbar_l}+o(\hbar_l|\log\hbar_l|^{-2}),\quad\\
z(\hbar_l)=\frac{1}{2}+o(1),&\quad \frac{\text{Im}\
  z(\hbar_l)}{\hbar_l}=o(|\log\hbar_l|^{-1})\,.
\end{split}
\end{equation}
To alleviate the notations we will omit the parameter $l$ and just use
$\hbar$, $z$, $\psi_\hbar$. A notable difference with the proof of
Theorem~\ref{t:maintheo1} is that we need to deal with quasimodes,
instead of eigenmodes (considering only eigenmodes would allow to prove the
inverse logarithmic gap~\eqref{e:gap}, but not the resolvent estimate~\eqref{e:resolv-est}). 

The assumptions~\eqref{e:contradiction} imply the following
estimates, that we will frequently use in our proof. For any $\mathcal{K}>0$, the following estimates hold uniformly
for times $|t|\leq \mathcal{K}|\log\hbar|$
\begin{align}\label{e:quasimodes}
\cU_\hbar^t e^{itz/\hbar}\psi_\hbar \defeq e^{-\frac{\imath  t}{\hbar}(\mathcal{P}(\hbar,z)-z)}\psi_{\hbar}&=\psi_{\hbar}+o(|t||\log\hbar|^{-2}),\quad
\text{and}\\
\label{e:small-Imz}
\quad e^{\frac{ t \Im\ z}{\hbar}}& =1+o(|t||\log\hbar|^{-1}).
\end{align}
Hence, even for $|t|\asymp|\log\hbar|$ both remainders
are $o_{\hbar\to 0}(1)$.

Applying the quasimode equation and~\eqref{e:quasimodes}, we obtain, for every fixed $t>0$,  
\begin{align*}
-\hbar^{-1}\,\Im z& =-\hbar^{-1}\,\Im\la\psi_{\hbar},\cP(\hbar,z)\psi_\hbar\ra+o(|\log\hbar|^{-1})\\
 & =\la\psi_{\hbar},a\,\psi_\hbar\ra+\cO(\hbar)+o(|\log\hbar|^{-1})\\
 & =e^{-2t\frac{\Im z}{\hbar}}\,
 \la\psi_{\hbar},(\U_\hbar^t)^*\,a\,\U_\hbar^t\psi_\hbar\ra + o_t(|\log\hbar|^{-1})\,.
\end{align*}
Applying the Egorov estimate \eqref{e:Egorov1}, in particular the case
described in remark~\ref{r:U*bU}, and
averaging over $t\in [-T,T]$, we get
$$
-\hbar^{-1}\,\Im z
=\big\la\psi_{\hbar},\Op_{\hbar}\Big(\frac{1}{2T}\int_{-T}^Ta\circ
g^te^{-2t\frac{\Im z}{\hbar}-2\int_0^ta\circ
  g^sds}dt\Big)\,\psi_\hbar\big\ra + o_T(|\log\hbar|^{-1}).
$$
Using the fact that the quantization procedure is almost positive --
see $\S$\ref{p:antiwick} -- and the identity \eqref{e:small-Imz}, one gets
the bound
$$
-\hbar^{-1}\,\Im z \geq
(1+o_T(1))e^{-2T\|a\|_{\infty}}\big\la\psi_{\hbar},\Op_{\hbar}\Big(\frac{1}{2T}\int_{-T}^Ta\circ
g^tdt\Big)\,\psi_\hbar\big\ra + o_T(|\log\hbar|^{-1}).
$$
We now use the cutoff function $P_{\infty}\in C^\infty_c(V_\infty,[0,1])$ introduced in
$\S$\ref{s:discretization}: notice that its
support is at positive distance
from $\cN^\delta$. 
Using again that $\Op_{\hbar}$ is approximately positive, one finds that
$$
-\hbar^{-1}\,\Im z \geq
(1+o_T(1))e^{-2T\|a\|_{\infty}}\big\la\psi_{\hbar},\Op_{\hbar}\Big(P_{\infty}\times\frac{1}{2T}\int_{-T}^Ta\circ
g^tdt\Big)\,\psi_\hbar\big\ra + o_T(|\log\hbar|^{-1}).$$
Since $P_{\infty}$ is supported away from the undamped set
$\cN^\delta$, there exists $T>0$ and $a_0>0$ (independent of $\hbar$) such that
$$\inf_{\rho\in\text{supp}P_\infty}\frac{1}{2T}\int_{-T}^Ta\circ g^t(\rho)dt\geq a_0,$$
which implies
$$
-\hbar^{-1}\,\Im z \geq (1+o_T(1))a_0
\,e^{-2T\|a\|_{\infty}}\left\la\psi_{\hbar},\Op_{\hbar}(P_{\infty})\,\psi_\hbar\right\ra
 + o_T(|\log\hbar|^{-1}).
$$
In particular, from our assumption on $\Im z(\hbar)$ we get
\begin{equation}\label{e:concl}
\la\psi_{\hbar},\Op_{\hbar}(P_{\infty})\,\psi_\hbar\ra = o(|\log\hbar|^{-1})\,.
\end{equation}
To obtain a contradiction we will prove an inverse logarithmic
lower bound for the above left hand-side. This can be achieved
by adapting the argument of Theorem~\ref{t:maintheo1}.

We will use the notations introduced in $\S$\ref{s:discretization}. 
Instead of considering the subset of cylinders $\Lambda_n\subset W^n$ in the argument of
$\S$\ref{ss:proof}, 
we will use the full family $W^n$, and obtain an upper bound for
$$
\big\|\sum_{\gamma\in
    (W^n)^c}\tPi_{\gamma}\psi_{\hbar}\big\|,
$$
where $(W^n)^c$ is the complementary of $W^n$ in $\overline{W}^n$. 
Recall that $n=[\kappa_0|\log\hbar|]$ is a short logarithmic time, for
which we may apply Egorov's Theorem and the pseudodifferential calculus. 

\begin{rema} In \S\ref{s:proof} the restriction to cylinders in $\Lambda_n$ had allowed to show that the
Birkhoff averages $-\int_0^{nn_0}a\circ g^s(\rho)ds$ were bounded above by $\beta
nn_0+\mathcal{O}(1)$, a property which was crucially used in
\S\ref{pa:submult}. 
We are now interested in the case $\beta=0$, and 
the upper bound $-\int_0^{nn_0}a\circ g^s(\rho)ds\leq 0$ obviously holds for
every point $\rho\in T^*M$ since $a$ is nonnegative. 
\end{rema}

Using the hyperbolic
dispersive estimate \eqref{e:HDE1} and taking the sum over $W^n$, we
can prove the inequality
\eqref{e:proofstep2} for our quasimode
$\psi_\hbar$. Using~\eqref{e:small-Imz} and the fact 
that the time
$knn_0=\cO(|\log\hbar|)$, we get
$$
\Big\|\sum_{\Gamma\in(W^{nk})^c}\Pi_{\Gamma}\,\psi_{\hbar}\Big\|\geq
 1+o(1)\,,
$$
Implementing the same subadditivity
argument as in \S\ref{pa:submult}, we find
$$ 1+o(1)\leq 
\Big\|\sum_{\Gamma\in(W^{nk})^c}\Pi_{\Gamma}\,\psi_{\hbar}\Big\|\leq
c^k\,  \,(1+o(1))\sum_{j=0}^{k-1}\,\Big\|\sum_{\gamma\in
  (W^n)^c}\Pi_{\gamma}\U_{\hbar}^{jn}\,\psi_{\hbar}\Big\|.$$
Thanks to the upper bound~\eqref{e:calderonbis} and the subunitarity
bound $\|\U_{\hbar}^{n}\|\leq 1$, we verify that 
$\Big\|\sum_{\gamma\in (W^n)^c}\Pi_{\gamma}\Big\|=\cO(1)$. We now use
the identities~(\ref{e:quasimodes},\ref{e:small-Imz}) one more time and we obtain
$$ 1+o(1)\leq
c^k\, k \,(1+o(1))\,\Big\|\sum_{\gamma\in
  (W^n)^c}\Pi_{\gamma}\,\psi_{\hbar}\Big\|+o(1).$$
Like in \S\ref{pa:semiclapprox} and using again~(\ref{e:quasimodes},\ref{e:small-Imz}), this inequality can be rewritten as
$$
(c^k\, k)^{-1} \big(1+o(1)\big)\leq  \Big\|\sum_{\gamma\in (W^n)^c}\tPi_{\gamma}\,\psi_{\hbar}\Big\|\,,
$$
and then analyzed through the pseudodifferential calculus like in the proof of Theorem~\ref{t:maintheo1}. We obtain\footnote{Like in paragraph~\ref{pa:semiclapprox}, the parameter $\nu_0>0$ will change from line to line, meaning that we keep the worst remainder term.} 
\begin{equation}\label{e:lower3}
(c^kk)^{-2}(1+o(1)) \leq  \big\langle\Op_{\hbar}\Big(\sum_{\gamma\in (W^n)^c}\tP_{\gamma}\,\Big)\psi_{\hbar},\psi_{\hbar}\big\rangle+\mathcal{O}(\hbar^{\nu_0}).
\end{equation}
The set $(W^n)^c$ consists in the cylinders in $\overline{W}^n$ with at
least one index $\gamma_j=\infty$, so it can be split into
$$
(W^n)^c = \bigsqcup_{p=1}^n \{\Gamma=\overline{\gamma}\,\infty\,\gamma\, :\, \overline{\gamma}\in
  \overline{W}^{p-1},\,\gamma\in W^{n-p}\}\,.
$$
Accordingly,
$$\sum_{\gamma\in (W^n)^c}\tP_{\gamma} 
=\sum_{p=1}^n \Big(\big(\sum_{\gamma\in W^{n-p}}\tP_{\gamma}\big)
P_{\infty}\circ g^{-(n-p+1)n_0}\, \big(\sum_{\overline{\gamma}\in
  \overline{W}^{p-1}} \tP_{\overline{\gamma}}\circ g^{-(n-p+1)n_0}\big)\Big).
$$
Since
the family $(P_{\alpha})_{\alpha\in\overline{W}}$ forms a resolution
of identity near $\cE^{\delta/2}$, we have for any $t\in\IR$
$$\sum_{{\gamma}\in W^{n-p}} \tP_\gamma\circ g^t \leq 1,\quad
\sum_{\overline{\gamma}\in \overline{W}^{p-1}} \tP_{\overline{\gamma}}\circ g^t = 1,\quad 
 \text{near }\cE^{\delta/2}\,.
$$
The approximate positivity of $\Oph$ implies
$$
\big\la\Oph\big(\sum_{\gamma\in (W^n)^c}\tP_{\gamma}\big)
\psi_\hbar,\psi_\hbar\big\ra \leq 
\sum_{p=1}^n  \big\la \Oph(P_{\infty}\circ g^{(p-n-1)n_0})\psi_\hbar,\psi_\hbar\big\ra + \cO(\hbar^{\nu_0}),
$$
so from \eqref{e:lower3} we get
$$(c^kk)^{-2}(1+o(1)) \leq
\sum_{p=1}^{n}\left\langle\Op_{\hbar}\left(P_{\infty}\circ
    g^{-pn_0}\right)\psi_{\hbar},\psi_{\hbar}\right\rangle+\mathcal{O}(\hbar^{\nu_0}).
$$ 
We now again combine the fact that $\psi_{\hbar}$ is an quasimode (via equation~\eqref{e:quasimodes}) with
the Egorov theorem, and obtain
$$
(c^kk)^{-2}(1+o(1)) \leq
\sum_{p=1}^{n}\big\langle\Op_{\hbar}\big(P_{\infty}\,e^{-2pn_0\frac{\Im z}{\hbar}-2\int_0^{pn_0}a\circ
      g^sds}\big)\psi_{\hbar},\psi_{\hbar}\big\rangle+\mathcal{O}(\hbar^{\nu_0})+o(n^2|\log\hbar|^{-2}).
$$
A last application of the fact that $a\geq 0$, $\Im
z=o(\hbar|\log\hbar|^{-1})$, $n=\cO(|\log\hbar|)$ and that $\Op_{\hbar}$ is almost positive
implies that
$$
(c^kk)^{-2}(1+o(1)) \leq 
n\, (1+o(1))\, 
\la\Op_{\hbar} (P_{\infty}) \psi_{\hbar},\psi_{\hbar}\ra+\mathcal{O}(\hbar^{\nu_0})+o(1).
$$
Hence, for $n=[\kappa_0|\log\hbar|]$ we end up with
$$
\frac{(c^kk)^{-2}}{\kappa_0|\log\hbar|}(1+o(1)) \leq  \la\Oph(P_{\infty})\psi_{\hbar},\psi_{\hbar}\ra.$$
This lower bound establishes the contradiction with
Eq.~\eqref{e:concl}, and shows that our
assumption~\eqref{e:contradiction} cannot be verified. This proves the
resolvent estimate~\eqref{e:resolvent}, and our theorem.
\end{proof}

\begin{rema}
Provided that we consider a sequence of $o(\hbar|\log\hbar|^{-2})$ quasimodes, the above logarithmic lower bound on
$\la\psi_{\hbar},\Op_{\hbar}\left(P_{\infty}\right)\,\psi_\hbar\ra$
holds as well in the selfadjoint 
case for a smooth cutoff function $1-P_{\infty}$ around an hyperbolic subset $\Lambda$ satisfying $P_{top}(\Lambda,g^t,\log J^u/2)<0$. In fact, its proof only used the fact that $\Im
z=o(\hbar|\log\hbar|^{-1})$ and $a\geq 0$. In this case, this lower
bound generalizes the concentration results obtained in~\cite{CdVPa94,
  ToZe03, BuZw04, Chr07} for hyperbolic closed geodesics (yet, the required precision of our quasimode is
stronger than the one used in~\cite{Chr07}; besides, our result does
not encompass the case of a semihyperbolic orbit treated
in~\cite{Chr11}).
\end{rema}


\begin{thebibliography}{99}
\bibitem{An08} N.~Anantharaman \emph{Entropy and the localization of eigenfunctions}, Ann. of Math. (2) $\mathbf{168}$, 438--475 (2008)
\bibitem{An10a} N.~Anantharaman \emph{Spectral deviations for the damped wave equation}, Geom. Func. Anal. $\mathbf{20}$, 593--626 (2010)
\bibitem{An10b} N.~Anantharaman \emph{A hyperbolic dispersion estimate, with applications to the linear Schrodinger equation}, 
Proceedings of the International Congress of Mathematicians 2010, Vol. III (2010)
\bibitem{AN07} N.~Anantharaman, S.~Nonnenmacher \emph{Half delocalization of eigenfunctions of the Laplacian on an Anosov manifold}, Ann. Inst. Fourier $\mathbf{55}$, 2465--2523 (2007)
\bibitem{AsLeb03} M.~Asch, G.~Lebeau \emph{The spectrum of the damped wave operator for a bounded domain in $\mathbb{R}^2$}, Exp. Math. $\mathbf{12}$, 227--240 (2003)
\bibitem{BoRo02} A.~Bouzouina, D.~Robert \emph{Uniform semiclassical estimates for the propagation of quantum observables}, Duke Math. Jour. $\mathbf{111}$, 223--252 (2002)
\bibitem{BoRue75} R.~Bowen, D.~Ruelle \emph{The ergodic theory of Axiom A flows}, Inv. Math. $\mathbf{29}$, 181-202 (1975)
\bibitem{Bu97} N.~Burq \emph{Mesures semi-classiques et mesures de d\'efaut (d'apr\`es P.~G\'erard, L.~Tartar et al.)}, Ast\'erisque $\mathbf{245}$, 167--196, S\'eminaire Bourbaki, (1996-1997)
\bibitem{BuChr09} N.~Burq, H.~Christianson \emph{Imperfect control for
    the damped wave equation}, Comm. Math. Phys. $\mathbf{336}$, 101--130  (2015)
\bibitem{BuZw04} N.~Burq, M.~Zworski \emph{Geometric control in the presence of a black box}, J. Amer. Math. Soc. $\mathbf{17}$, 443--471 (2004)
\bibitem{Chr07} H.~Christianson \emph{Semiclassical nonconcentration
    near hyperbolic orbits}, J. Funct. Anal. $\mathbf{246}$, 145--195
  (2007); \emph{Corrigendum to ``Semiclassical nonconcentration near
    hyperbolic orbits''}, J. Funct. Anal. $\mathbf{258}$, 1060--1065
  (2009)
\bibitem{Chr09}  H.~Christianson {\em Applications of Cutoff Resolvent
    Estimates to the Wave Equation}, Math. Res. Lett. {\bf 16}
  577--590 (2009)
\bibitem{Chr11}  H.~Christianson \emph{Quantum Monodromy and
    Non-concentration Near a Closed Semi-hyperbolic Orbit},
  Trans. Amer. Math. Soc. $\mathbf{363}$, 3373--3438  (2011)
\bibitem{CSVW12}  H.~Christianson, E.~Schenck, A.~Vasy, J.~Wunsch \emph{From resolvent estimates to damped waves},
 J. Anal. Math. $\mathbf{122}$, 143--162 (2014) 
\bibitem{CdVPa94} Y.~Colin de Verdi\`ere, B.~Parisse \emph{\'Equilibre instable en r\'egime semi-classique. I. Concentration microlocale}, CPDE $\mathbf{19}$, 1535--1563 (1994)
\bibitem{DS} M.~Dimassi, J.~Sj\"ostrand \emph{Spectral Asymptotics in the Semiclassical Limit} Cambridge University Press (1999)
\bibitem{HeMaRo} B.~Helffer, A.~Martinez, D.~Robert \emph{Ergodicit\'e
    et limite semi-classique},  Commun. Math. Phys. $\mathbf{109}$, 313-326 (1987)
\bibitem{Hi02} M.~Hitrik \emph{Eigenfrequencies for Damped Wave Equations on Zoll manifolds}, Asympt. Analysis $\mathbf{31}$, 265--277 (2002)
\bibitem{Hi04} M.~Hitrik \emph{Eigenfrequencies and Expansions for Damped Wave Equations}, Methods
and Applications of Analysis $\mathbf{10}$, 543--564 (2003)
\bibitem{Ho} L.~H\"ormander \emph{The Analysis of Linear Partial Differential Operators III}, Springer-Verlag, Berlin, New York (1985)
\bibitem{KaHa} A.~Katok, B.~Hasselblatt \emph{Introduction to the Modern Theory of Dynamical Systems}, Cambbridge University Press (1995)
\bibitem{Leb93} G.~Lebeau \emph{\'Equation des ondes amorties}, Algebraic and geometric methods in mathematical physics (Kaciveli 1993), Math. Phys. Stud. $\mathbf{19}$, 73--109 (1996)
\bibitem{Non11} S.~Nonnenmacher \emph{Spectral theory of damped quantum chaotic systems}, Journ\'ees \'equations aux d\'eriv\'ees partielles, Exp. No. 9, avalaible at http://jedp.cedram.org/ (2011)
\bibitem{NZ09} S.~Nonnenmacher, M.~Zworski \emph{Quantum decay rates in chaotic scattering},  Acta Math. $\mathbf{203}$, 149--233 (2009)
\bibitem{Pe} Y.~Pesin \emph{Dimension Theory in Dynamical Systems: Contemporary Views and Applications}, The University of Chicago Press, Chicago (1998)
\bibitem{Riv11a} G.~Rivi\`ere \emph{Delocalization of slowly damped eigenmodes on Anosov manifolds}, Comm. in Math. Phys. in press (2012)
\bibitem{Roy10} J.~Royer \emph{Analyse haute fr\'equence de l'\'equation de Helmholtz disipative}, PhD Thesis, Universit\'e de Nantes, avalaible at 
http://tel.archives-ouvertes.fr/tel-00578423/fr/ (2010)
\bibitem{Sch10} E.~Schenck \emph{Energy decay for the damped wave equation under a pressure condition}, Comm. Math. Phys. $\mathbf{300}$, 375--410 (2010)
\bibitem{Sch11} E.~Schenck \emph{Exponential stabilization without geometric control}, Math. Research Letters $\mathbf{18}$, 379--388 (2011)
\bibitem{Sj00} J.~Sj\"ostrand \emph{Asymptotic distributions of eigenfrequencies for damped wave equations}, Publ. RIMS $\mathbf{36}$, 573--611 (2000)
\bibitem{ToZe03} J.~A.~Toth, S.~Zelditch \emph{$L^p$ norms of eigenfunctions in the completely integrable case}, Ann. H. Poincar\'e $\mathbf{4}$, 343--368 (2003)
\bibitem{Ze09} S.~Zelditch \emph{Recent developments in mathematical quantum chaos}, Current Developments in Mathematics, International Press of Boston, 115--202 (2009)
\bibitem{EZ} M.~Zworski \emph{Semiclassical analysis}, Graduate Studies in Mathematics $\mathbf{138}$, AMS (2012)
\end{thebibliography}
\end{document}